\documentclass[reqno]{amsart}

\usepackage{amssymb,amsmath,amscd, amsthm,xspace,color,tocvsec2, array,
	tikz-cd, aligned-overset}
\usepackage[all, cmtip]{xy}
\usepackage[mathscr]{euscript}
\usepackage{enumerate}
\usepackage[breaklinks=true]{hyperref}%To get clickable links to papers
\usepackage[letterpaper]{geometry}
\usepackage[all, cmtip]{xy}
\usepackage[ampersand]{easylist}
\usepackage{stmaryrd}
\usepackage[myheadings]{fullpage}
\usepackage{csquotes}
\usepackage{mathtools}
	
	\newcommand{\fk}{\mathfrak{k}}

\newcommand{\g}{\widehat{\mathfrak{g}}}

\newcommand{\fu}{\mathfrak{u}}

\newcommand{\halpha}{\check{\alpha}}

\newcommand{\fsl}{\mathfrak{sl}}

\newcommand{\ft}{\mathfrak{t}}

\newcommand{\fg}{\mathfrak{g}}

\newcommand{\comment}[1]{}

\newcommand{\OO}{\mathscr{O}}

\newcommand{\fb}{\mathfrak{b}}
\newcommand{\gk}{\g_\kappa}
\newcommand{\fn}{\mathfrak{n}}

\newcommand{\sD}{\mathscr{D}}

\newcommand{\on}{\operatorname}

\newcommand{\ftp}{D_\kappa}

\newtheorem{theo}[subsubsection]{Theorem}
\newtheorem*{theo*}{Theorem}

\newtheorem{lemma}[subsubsection]{Lemma}
\newtheorem{cor}[subsubsection]{Corollary}
\newtheorem{pro}[subsubsection]{Proposition}

\theoremstyle{remark}

\newtheorem{re}[subsubsection]{Remark}

\newtheorem{ex}[subsubsection]{Example}

\newcommand{\hs}{\hspace{1mm}}

\newcommand{\rhoc}{\check{\rho}}

\numberwithin{equation}{section}

\newcommand{\Wf}{{}^fW}

\newcommand{\fh}{\mathfrak{h}}
\newcommand{\sC}{\mathscr{C}}

\newcommand{\msf}{\operatorname}

\newcommand{\sW}{\EuScript{W}_\kappa}
\newcommand{\cW}{\mathscr{W}}
\newcommand{\Avr}{\on{Av}_{*}^{\psi, r}}
\newcommand{\sI}{\hs \overset I \star \hs }
\newcommand{\crho}{\check{\rho}}

\renewcommand{\mod}{\operatorname{-mod}}

\newcommand{\DGCat}{\msf{DGCat}_{\on{cont}}}

\begin{document}

%\frenchspacing

\title[]{The Drinfeld--Sokolov reduction of admissible representations of affine Lie algebras}

\author{Gurbir Dhillon}
\address{Yale University \\ New Haven, CT
06511, USA}

\email{gurbir.dhillon@yale.edu}

\date{\today}

\begin{abstract} Fix an affine Lie algebra $\gk$ with associated principal affine W-algebra $\sW$. A basic conjecture of Frenkel--Kac--Wakimoto asserts that Drinfeld--Sokolov reduction sends admissible $\gk$-modules to zero or cohomological shifts of minimal series $\sW$-modules. We prove this conjecture and a natural generalization to the spectrally flowed Drinfeld--Sokolov reduction functors. This extends the  previous results of Arakawa for the minus reduction. \end{abstract}

\maketitle

\section{Introduction}

\subsection{} 
In the seminal work \cite{ff90} Feigin and Frenkel introduced the affine W-algebras into mathematics. These algebras and their representation theory have been intensively studied in the ensuing three decades and have fundamental relations to integrable hierarchies, moduli spaces associated to algebraic curves, and the geometric Langlands program, cf. \cite{araintro17}, \cite{fbz}, \cite{Frbook07} and references therein. 

In this paper we settle a basic conjecture from \cite{ff90}, or more precisely a refinement due to Frenkel--Kac--Wakimoto \cite{fkw}. In the remainder of the introduction we review some of the history and context of the conjecture; the reader may wish to skip directly to Section \ref{s:sres} where we state our results.

\subsection{} The method by which Feigin and Frenkel  produced the W-algebras is cohomological in nature, and takes as its input a simple Lie algebra $\fg$ equipped with an adjoint-invariant bilinear form $\kappa$. Let us write $\gk$ for the affine Lie algebra associated to $\fg$ and $\kappa$, which is a central extension of the loop algebra $\fg(\!(t)\!)$: 
\[
 0 \rightarrow k \cdot \mathbf{1} \rightarrow \gk \rightarrow \fg(\!(t)\!) \rightarrow 0.
\]
By the construction of the associated W-algebra, which we denote by $\sW$, one has a functor of {Drinfeld--Sokolov reduction} between the derived categories of modules
\begin{equation} \label{e:dsr}
   \Psi^+: \gk\mod \rightarrow \sW\mod. 
\end{equation}

Let us mention two basic properties of this functor. First, it sends the vacuum Kac--Moody vertex algebra $\mathbb{V}_{\kappa}$, concentrated in cohomological degree zero, to $\sW$, concentrated in cohomological degree zero. Second, it is constructed via conformal field theoretic techniques. As a consequence, the functor {\em factorizes}. That is, for any smooth algebraic curve $X$, Equation \eqref{e:dsr} naturally extends to a functor between the  versions of the appearing categories of chiral algebra modules defined over the Ran space of $X$, cf. \cite{bdchiral04}.  

\subsection{} In the final pages of \cite{ff90}, Feigin and Frenkel proposed a conjectural property of Drinfeld--Sokolov reduction. Kac and Wakimoto had discovered for any positive rational level $\kappa$ a family of irreducible highest weight $\gk$-modules, the {admissible} modules, distinguished by the modularity properties of their characters \cite{kw88}, \cite{kw89}. Feigin and Frenkel predicted that Drinfeld--Sokolov reduction should send admissible modules of $\gk$ to zero or irreducible {minimal series} representations of $\sW$. The latter had attracted interest in both mathematics and physics due to their appearance in rational conformal field theory, cf. \cite{bouwknegt93survey}.

In particular, by an analysis of Euler characteristics, combined with the  character formula for admissible modules due to Kac--Wakimoto, the
conjecture would imply formulae for the $q$-characters of the minimal series, which was the subject of forthcoming joint work of Frenkel--Kac--Wakimoto \cite{fkw}.

\begin{ex} Let us specialize the preceding discussion to the simplest case of $\fg = \fsl_2$. Here, the algebra $\sW$ is a Virasoro vacuum algebra, whose central charge varies with $\kappa$, and its minimal series are the celebrated minimal models of chiral conformal field theory \cite{bpz84}. As Feigin and Frenkel emphasize in \cite{ff90}, the correspondence with admissible $\widehat{\fsl}_2$-modules had already been proposed and studied in the physical literature \cite{kpz}, \cite{mukhipanda}, and proven by the authors themselves in \cite{ff902}. 
\end{ex}

\subsection{} In the classic paper \cite{fkw}, Frenkel--Kac--Wakimoto performed a comprehensive analysis of the resulting $q$-characters and formulated a series of conjectures based on their findings. Combining these with the Verlinde formula, they obtained a conjectural calculation of the fusion rules for the minimal series. We now recall two important features of their work.  

A first basic discovery of Frenkel--Kac--Wakimoto was that, beyond the case of $\fg = \fsl_2$, Drinfeld--Sokolov reduction could not simply send an admissible module, concentrated in cohomological degree zero, to a minimal series representation, concentrated in cohomological degree zero. Instead, to account for certain signs encountered in their analysis of characters, they conjectured that admissible modules are sent to zero or cohomological shifts of minimal series. Moreover, they gave for each admissible module a prediction for whether its reduction vanished, and if not the exact cohomological shift and the highest weight for the obtained minimal series representation. These conjectures intricately depend on the level $\kappa$ and the highest weight of the admissible module. They have been since regarded as the definitive formulation of the prediction announced by Feigin--Frenkel on Drinfeld--Sokolov reduction.  

A second basic contribution of Frenkel--Kac--Wakimoto was their introduction of another functor
\begin{equation} \label{e:mdsr}
  \Psi^-: \gk\mod \rightarrow \sW\mod, 
\end{equation}
which they called the minus Drinfeld--Sokolov reduction. In the case of $\fg = \fsl_2$, this had been previously considered by Feigin--Frenkel \cite{ff902}. This functor, which is a variant of the usual, i.e. plus, Drinfeld--Sokolov reduction, is less immediate to construct using conformal field theoretic techniques. Relatedly, it no longer exchanges the vacuum vertex algebras, and does not extend to a factorizable functor.

However, a crucial realization of Frenkel--Kac--Wakimoto was that the minus reduction should again send admissible modules to zero or minimal series, and in a drastically simpler fashion than the plus reduction. More precisely, they conjectured that the reduction of an admissible series representation would not vanish if and only if its highest weight, considered as a weight of the abstract Cartan of the finite dimensional algebra $\fg$, was antidominant. They further conjectured that the resulting $\sW$-module would always be a minimal series representation, concentrated in cohomological degree zero. 

 The findings of Frenkel--Kac--Wakimoto, and in particular their conjectures on the plus and minus reductions, have proven highly influential. They continue to guide developments in the subject, e.g. the  emerging connections between minimal models and the quantum Langlands correspondence. We note that for the latter connection, and more generally questions of a global nature involving conformal blocks, the factorizability of the plus reduction plays an important role. 

\subsection{} After both conjectures lay open for over a decade, a breakthrough  was made by Arakawa. In the tour-de-force works \cite{ara04}, \cite{ara07}, he completely proved the conjecture for the minus reduction, and verified many cases of the conjecture for the plus reduction. His work moreover yielded character formulae for every simple highest weight $\sW$-module, extending the formulae for minimal series conjectured by Frenkel--Kac--Wakimoto. We remark that an alternative proof of the conjecture for the minus reduction was given by the author and Raskin recently in the works \cite{whit}, \cite{locw}.

\subsection{} The decade following Arakawa's work has seen tremendous progress in the study of W-algebras. Notably, many of the main conjectures from the early days of the subject, e.g. the relation with coset models and the rationality of minimal quotients, have been fully resolved  \cite{ara15}, \cite{ara19}. However, to our knowledge, no further cases of Frenkel--Kac--Wakimoto's conjecture on the plus reduction were obtained.\footnote{However, some cases of the analogous assertions for other spectrally flowed reduction functors may be found in the interesting work \cite{acf}. }

The basic added difficulty for the conjecture on plus reduction can be understood immediately from its statement. Namely, the absence of cohomological shifts in the conjecture for the minus reduction underlies the fact that the functor is $t$-exact on all of the Category $\mathscr{O}$ of $\gk$-modules. This was known for $\fg = \fsl_2$ by the work of Feigin--Frenkel \cite{ff902}, and was proven in general by Arakawa \cite{ara04}. This $t$-exactness plays a crucial role in the argument of Arakawa for the minus reduction, and that of the author and Raskin. An analogous $t$-exactness holds for the plus reduction applied to certain blocks of Category $\OO$, which were exactly the cases settled by Arakawa. Beyond these cases, the plus reduction is not $t$-exact, so one needs a further idea to proceed.    

\subsection{} In this paper, we prove Frenkel--Kac--Wakimoto's conjecture for the plus reduction, as well as a natural generalization to arbitrary dominant weights and spectrally flowed reduction functors.

\section{Statement of Results}
\label{s:sres}
 
In this section, after introducing the necessary notation and preliminaries, we state our main Theorem \ref{t:maintheorem}.

\subsection{} \label{ss:lambda}Fix an arbitrary level $\kappa$, and let $L_\lambda$ be an irreducible highest weight $\gk$-module which is dominant. That is, we suppose its Verma cover $M_\lambda$ is projective in the Category $\OO$. In particular $L_\lambda$ could be any admissible module. 

 \subsection{}\label{ss:mu} \label{s:minusred}Fix in addition an element $\check{\mu}$ of the coweight lattice of the finite Cartan. One may associate to it the spectrally flowed Drinfeld--Sokolov reduction functor 
\begin{equation} \label{e:cmured}
     \Psi^{\check{\mu}} : \gk\mod \rightarrow \sW\mod,  
\end{equation}
cf. Section \ref{s:sfint} for the definition, which is slightly nonstandard if $\check{\mu}$ is not antidominant. For now, it  suffices to say that if we write $\check{\rho}$ for one half of the sum of the positive finite coroots,  then the functor $\Psi^{\crho}$ equals the plus reduction, and the functor $\Psi^0$ equals the minus reduction.

\subsection{} Our main theorem calculates the spectrally flowed reductions
\[
 \Psi^{\check{\mu}} (L_\lambda) \in \sW\mod. 
\]
If we write $G$ for the simple algebraic group of adjoint type with Lie algebra $\fg$, and $LG$ for its algebraic loop group, then as we explain in Section \ref{s:sfint} this calculation is equivalent to computing, for all $g \in LG$, the plus reduction of the $\on{Ad}_g$-conjugated simple modules 
\[
  \Psi^+\big(\hspace{-.5mm}\on{Ad}_g(L_\lambda)\big). 
\]
The answer involves the following Coxeter-theoretic combinatorics. 

Let us denote the finite Weyl group, the extended affine Weyl group, and the integral Weyl group of $L_\lambda$ by $W_f, W,$ and $W_\lambda,$ respectively. In particular, we have the tautological inclusions
\[
   W_f \subset W \supset W_\lambda. 
\]
%
%These are all Coxeter systems, i.e. Coxeter groups equipped with distinguished sets of simple reflection. 

Write $e^{\check{\mu}}$ for the element of $W$ corresponding to $\check{\mu}$. Let us say that the pair $(\check{\mu}, \lambda)$ is {\em good} if the double coset 
\begin{equation} \label{e:doubs}
    W_f e^{\check{\mu}} W_\lambda
\end{equation}
is a $W_f \times W_\lambda$ torsor, i.e. the action is simply transitive. In this case one has a unique factorization 
\begin{equation} \label{e:cannonfact}
 e^{\check{\mu}} = w_f w_- w_\lambda,
\end{equation}
where $w_f \in W_f$, $w_-$ is the unique minimal length element of \eqref{e:doubs} with respect to the Bruhat order on $W$, and $w_\lambda \in W_\lambda$. We write $\ell_\lambda(w_\lambda)$ for the length of $w_\lambda$ viewed as an element of the Coxeter system $W_\lambda$.

\subsection{} Finally, we need some notation. Write $\ft^*$ for the dual of the finite Cartan. Recall that the combinatorics of highest weight $\gk$-modules is governed by the level $\kappa$ dot action of $W$ on $\ft^*$, which we denote by 
\[
     W \times \ft^* \rightarrow \ft^*, \quad \quad (w, \lambda) \mapsto w\lambda,
\]
cf. Section \ref{sss:dotaction} for further discussion of this action. To discuss highest weights for $\sW$-modules, we denote the Harish--Chandra projection by
\[
 \pi: \ft^* \rightarrow \ft^* /\hspace{-1.1mm}/ W_f,
\]
and for an infinitesimal character 
$
   \chi \in \ft^* /\hspace{-1.1mm}/ W_f
$
we write $L_\chi$ for the corresponding simple highest weight $\sW$-module, cf. Section \ref{ss:hwminus} for our precise normalization.

\subsection{} With these preparations, we can state our main result.

\begin{theo}\label{t:maintheorem} Let $\check{\mu}$ and $\lambda$ be as in Sections \ref{ss:mu} and \ref{ss:lambda}, respectively. If the pair $(\check{\mu}, \lambda)$ is good, let $w_-$ and $w_\lambda$ be as in Equation \eqref{e:cannonfact}. Then the corresponding spectrally flowed reduction is given by 
\[
\Psi^{\check{\mu}}(L_\lambda) \simeq \begin{cases} L_{ \pi(w_- \lambda)}[\ell_\lambda(w_\lambda)] & \text{if $(\check{\mu}, \lambda)$ is good}, \\ 0 & \text{otherwise.} \end{cases} 
\]
\end{theo}

We would like to make two comments about the theorem.

\begin{re} Let us take $\check{\mu} = \check{\rho}$, so that we are considering the plus reduction, and let us further suppose that $\lambda$ is a principal admissible weight in the sense of Kac--Wakimoto \cite{kw89}. This case of Theorem \ref{t:maintheorem} confirms Conjecture $\on{3.4}_+$ of Frenkel--Kac--Wakimoto \cite{fkw}. 

More precisely, (i) whether the reduction vanishes is exactly as they conjectured. Further, (ii) when the reduction is nonvanishing we get the irreducible $\sW$-module that they conjectured concentrated in a single cohomological degree. Finally, (iii) the exact cohomological degree corrects their conjecture. Namely, we find a shift by the length of $w_\lambda$ rather than minus its semi-infinite length; these differ in examples already for $\fg = \fsl_3$. Note however that this does not affect the parity of the cohomological degree, and in particular is consistent with their Euler characteristic calculations. 
\end{re}

\begin{re} After stating Conjecture $\on{3.4}_+$, Frenkel--Kac--Wakimoto outline in Conjectures 3.5.1 and 3.5.2 a conjecture on the cohomology of the plus reduction of a 2-sided resolution of $L_\lambda$ by Wakimoto modules. Their Conjecture 3.5.1 on the existence of two sided resolutions was settled by Arakawa \cite{arabgg}, and Conjecture 3.5.2 then follows by Theorem \ref{t:maintheorem}. 
\end{re}

\subsection{} The method by which we prove Theorem \ref{t:maintheorem}, which we will indicate in the next subsection, also yields tight bounds on the cohomological amplitude of any $\Psi^{\check{\mu}}$ when restricted to a block of Category $\OO$. The right exactness of the functors follows from work of Frenkel--Gaitsgory \cite{fg06loc}. Our lower bounds on the nonvanishing cohomology are to the best of our knowledge new save for the case of the minus reduction, where they reduce to its $t$-exactness. 

For example, for the block containing $L_\lambda$ as above, we show the functor has cohomological amplitude at most $[-\ell_\lambda(w_\lambda), 0]$. That is, the reduction of $L_\lambda$, when nonzero, is concentrated in the most negative possible degree. The case of a general block is similar and is presented in Theorem \ref{t:cohamp} below.

\subsection{} Finally, let us describe the idea of the proof. Our starting point is the fact that the various reductions $\Psi^{\check{\mu}}$, when restricted to Category $\OO$,  differ from the minus reduction $\Psi^-$ by precomposition with intertwining, i.e. twisting, functors.\footnote{The twisting functors were constructed geometrically by Feigin--Frenkel in their fundamental work \cite{ffsif}, and  algebraically by Arkhipov in \cite{ark}. They also appear in the earlier work of Beilinson--Bernstein \cite{bb82} under the name intertwining functors.} Equivalently, writing $\mathring{I}$ for the prounipotent radical of $I$, they differ by convolution with costandard objects in the universal affine Hecke category 
\begin{equation} \label{e:biun}
   D_\kappa(\mathring{I} \backslash LG / \mathring{I}).
\end{equation}
Here universal refers to the fact that we impose no restrictions on the monodromies of the D-modules along the orbits of the finite Cartan $I/\mathring{I}$ under left and right multiplication. This relation with intertwining functors is an old observation going back at least to the work of Frenkel--Gaitsgory \cite{fg06loc} (and, after applying the sheaves-functions dictionary, far earlier still to the study of principal series representations of finite and $p$-adic Chevalley groups). 

In our approach to the conjecture we make essential use of the fact that, upon restriction to a fixed block $\OO_\chi$ of Category $\OO$, the relevant intertwining functors descend to costandard objects in an affine Hecke category with right monodromy fixed by $\chi$. As such, this carries an action by the bi $\chi$-monodromic affine Hecke category 
\begin{equation} \label{e:oneun} D_\kappa(\mathring{I} \backslash LG / I, \chi) \circlearrowleft D_\kappa(I, \chi \backslash LG / I, \chi).\end{equation}
This action categorifies the right action of the integral Weyl group on the extended affine Weyl group 
$$W \circlearrowleft W_\chi.$$
Practically, this means that the way in which costandard objects in \eqref{e:oneun} convolve with one another depends finely upon $\chi$. The upshot for us is that these convolutions satisfy identities which do not descend from the universal case \eqref{e:biun}.  

To prove Theorem \ref{t:maintheorem} we lift the factorization
\[
    e^{\check{\mu}} = w_f \circ w_- \circ w_\chi
\]
of Equation \eqref{e:cannonfact} to such an identity of costandard objects in \eqref{e:oneun}. That is, we write in Section \ref{s:sfint} a spectrally flowed reduction functor as the minus reduction precomposed with two intertwining functors:
\begin{equation} \label{e:3stepplan}
    \Psi^{\check{\mu}} \simeq \Psi^- \circ j_{w_-, *}^{w_- \chi, \chi} \circ j_{w_\chi, *}^{\chi, \chi}.\footnote{The precise definitions of the appearing intertwining functors, which are not necessary to follow this introductory discussion, nonetheless may be found in Section \ref{ss:twaffheck}. }
\end{equation}
We then track in Section \ref{s:fkw} how each step in the factorization transforms dominant highest weight modules in $\OO_\chi$ and find that it  produces the behavior anticipated by Frenkel--Kac--Wakimoto. Namely, the functors $\Psi^-$ and $j_{w_-, *}^{w_- \chi, \chi}$ of Equation \eqref{e:3stepplan} are $t$-exact and behave simply on the  abelian categories of modules; for $\Psi^-$ this is exactly Arakawa's theorem. However, the first term $j_{w_\chi, *}^{\chi, \chi}$ is typically far from $t$-exact, and it is in general difficult to calculate its action on a given object. However, on `maximally integrable' objects, such as dominant highest weight modules, it reduces to a cohomological shift of the identity functor, which yields the theorem. 

\begin{re} The cases of the plus reduction previously obtained by Arakawa \cite{ara07} are exactly the cases when $w_\chi$ is length zero. The latter condition is equivalent to $j_{w_\chi, *}^{\chi, \chi}$ being $t$-exact, and for this reason the arguments for the minus reduction, where one checks that (co)standard objects are sent to (co)standard objects, and deduce an assertion for arbitrary simples, carry over {\em mutatis mutandis}. \end{re}

In short, {we find the combinatorics of the Frenkel--Kac--Wakimoto conjecture for the plus reduction in the combinatorics of twisted affine Hecke categories}.

\subsection{} The paper is organized as follows. We gather in Section 3 some preliminary material on loop groups, W-algebras, categorical representation theory, Whittaker sheaves, and Hecke categories. 

In Section \ref{s:sfint} we relate the spectrally flowed reduction functors to Whittaker sheaves on the affine flag manifold and obtain the factorization of Equation \eqref{e:3stepplan}. 

Finally, in Section \ref{s:fkw} we apply this to obtain the main Theorem \ref{t:maintheorem}.

\vspace{.15cm}

{\noindent\bf Acknowledgments.} We warmly thank Sam Raskin for conversations and collaborations related to $\sW$-algebras, and in particular for introducing us to the conjecture addressed in this paper several years ago. We additionally thank him and Ivan Losev for comments on a draft of this paper, and Zhiwei Yun for some helpful conversations.

We thank Tomoyuki Arakawa, Thomas Creutzig, Edward Frenkel, and Geordie Williamson for correspondence related to the present work. 

The author was supported by an NSF Postdoctoral Fellowship under grant No. 2103387.

\section{Notation and preliminaries}

In this section we gather some known results and set notation. Aside from our convention that our fixed Borel $B$ is $TN^-$, and the opposite unipotent subgroup is $N$, our choices are more or less standard. So, the reader may wish to skip directly to Section \ref{s:sfint} and refer back only as needed.

\subsection{Loop groups and algebras} Some references for the following material are \cite{fbz}, \cite{frenkelloop}.

\subsubsection{} We work throughout over an algebraically closed field $k$ of characteristic zero. In particular, ind-schemes and their morphisms discussed below are over the spectrum of $k$. For an ind-scheme $X$ we denote by $X(k)$ its set of $k$-points. 

\subsubsection{} Let $G$ be a simple algebraic group of adjoint type.\footnote{The results of the paper for the W-algebra associated to a general connected reductive group follow immediately from this case.} We fix a Cartan subgroup $T$ and Borel subgroup $B = TN^-$, and write $N$ for the opposite maximal unipotent subgroup. We denote the respective Lie algebras by 
$
   \fg, \ft, \fb, \fn^-, \text{ and } \fn. 
$
Finally, we write $N_G(T)$ for the normalizer of $T$ in $G$, and denote the finite Weyl group by$$W_f :=  N_G(T) / T.$$
Recall this is a Coxeter system, i.e. a Coxeter group with a distinguished collection of simple reflections, which in this case is fixed by $B$. For an element $w \in W_f$, by a coset representative $\dot{w}$ we mean any lift to $N_G(T)$. 

\subsubsection{} Let us turn to the loop group. For an affine scheme $X$, we denote by $L^+X$ its formal arc space, which is a scheme of infinite type, and we denote by $LX$ its formal loop space, which is an ind-scheme of infinite type. Recall that there is a canonical closed embedding $L^+X \rightarrow LX$. As the formation of arc and loop spaces is functorial and commutes with finite products, they send affine group schemes to group ind-schemes. 

Within the loop group $LG$, our choice of $B \subset G$ distinguishes an Iwahori subgroup $I \subset LG$. Namely, $I$ is the preimage of $B$ under the evaluation map  
$
 L^+G \rightarrow G. 
$

\subsubsection{} We write $N_{LG}(LT)$ for the normalizer of $LT$ in $LG$, and denote the extended affine Weyl group by 
\[
    W := N_{LG}(LT)/L^+T. 
\]
For an element $w \in W$, by a coset representative $\dot{w}$ we mean any lift to $N_{LG}(LT)$. If we denote the cocharacter lattice of $T$, i.e. the coweight lattice, by $\check{\Lambda}$, we have a canonical isomorphism
\begin{equation} \label{e:affineweyl}
W \simeq W_f \ltimes \check{\Lambda}.
\end{equation}
Since there are different sign conventions for this isomorphism in the literature, let us spell this out carefully. For a cocharacter $\check{\mu}: \mathbb{G}_m \rightarrow T$, its restriction to the formal punctured disk about $0 \in \mathbb{A}^1(k)$ yields an element $t^{\check{\mu}} \in LT(k)$. 

With this, for an element $w \in W_f$, equipped with any lift $\dot{w}$ to $N_G(T)$, we identify the $L^+T$ coset of $\dot{w} t^{\check{\mu}}$ on the left-hand side of \eqref{e:affineweyl} with the element $we^{\check{\mu}}$ on the right-hand side. 

Let us also recall that if we write $\check{Q}$ for the coroot lattice, then the subgroup $W_f \ltimes \check{Q}$ is canonically a Coxeter system, with its simple reflections fixed by our Iwahori subgroup $I$. Its length function and Bruhat order naturally extend to all of $W$, and we denote these by $\ell$ and $<$ respectively.

\subsubsection{} Let us now discuss the loop algebras. Note that if $H$ is an affine algebraic group with Lie algebra $\fh$, then the topological Lie algebras of $L^+H$ and $LH$ are given by (the $k$-points of) $L^+ \fh$ and  $L\fh$, respectively. Explicitly, one has canonical isomorphisms 
\[
   L^+ \fh \simeq \fh \otimes k[\hspace{-.2mm}[t]\hspace{-.2mm}] \quad \text{and} \quad L \fh \simeq \fh \otimes k(\!(t)\!), 
\]
with the pointwise Lie bracket.

\subsubsection{} For a level, i.e. Ad-invariant bilinear form $\kappa$ on $\fg$, we denote the associated central extension of its loop algebra by 
\[
0 \rightarrow k \cdot \mathbf{1} \rightarrow \gk \rightarrow L \fg \rightarrow 0,
\]
and the associated vacuum vertex algebra by 
\[
 \mathbb{V}_\kappa \simeq \on{ind}_{L^+\fg \oplus k \cdot \mathbf{1}}^{\gk} k. 
\]
We would like to next review the definition of the associated affine W-algebra. 

\subsection{Semi-infinite cohomology and the $\sW$-algebra} The definition of the affine W-algebra is due to Feigin--Frenkel \cite{ff90}. Some references for the following material are \cite{vor}, \cite{fbz}, and \cite{araintro17}.

\subsubsection{} \label{sss:reldet}In this paper, we only need the following simple case of semi-infinite cohomology. Suppose $U$ is a unipotent algebraic group with Lie algebra $\fu$. For a smooth $L\fu$-module $M$ and a compact open subalgebra $\fk$ of $L\fu$, one may form the semi-infinite cohomology
\[
     C^{\frac{\infty}{2} + *}(L\fu, \fk, M) \in \on{Vect},
\]
where $\on{Vect}$ denotes the dg-category of unbounded complexes of $k$-vector spaces.\footnote{To set up some of the formalism we use and to construct some of the appearing categories, one has to work systematically with dg-categories. For this reason, we work with the canonical dg-enhancement of the derived category of vector spaces. However, for the purposes of this paper, one only really needs the existence of the associated triangulated categories throughout, so the reader may safely ignore such points.} 

Recall that $\fk$ appears to trivialize the dimension torsor for compact open subalgebras. Plainly, for another compact open subalgebra $\tilde{\fk}$ one has a canonical isomorphism 
\[
     C^{\frac{\infty}{2} + *}(L\fu, \tilde{\fk}, M) \otimes \on{rel. det}( \tilde{\fk},\fk)^* \simeq    C^{\frac{\infty}{2} + *}(L\fu, \fk, M),
\]
where in the appearing relative determinant line both $\fk$ and $\tilde{\fk}$ lie in cohomological degree -1.

\subsubsection{} Fix a generic additive character $\psi^-: N^- \rightarrow \mathbb{G}_a$. Associated to our coordinate $t$ is an induced additive character
\begin{equation} \label{e:res}
LN^- \xrightarrow{L\psi^-} L\mathbb{G}_a \xrightarrow{\on{Res}(-)} \mathbb{G}_a,
\end{equation} 
where, at the level of $k$-points, $\on{Res}(-)$ sends a formal power series $$f(t) \in k(\!(t)\!) \simeq L\mathbb{G}_a(k)$$to its coefficient of $t^{-1}$. We again use $\psi^-$ to denote the composition \eqref{e:res}, and we write $k_{\psi^-}$ for the induced one dimensional representations of $\fn^-$ and $L\fn^-$. With this, by definition the underlying vector space of $\sW$ is the complex
\[
     C^{ \frac{\infty}{2} + *}(L\fn^-, L^+\fn^-, \mathbb{V}_\kappa \otimes k_{\psi^-}),
\]
which is concentrated cohomologically in degree zero.

\subsection{Categories of smooth modules} 
\label{ss:gmodwmod}
\subsubsection{} For a dg-category $\sC$ equipped with a $t$-structure we denote by $\sC^+$ the full subcategory of bounded below objects. Similarly, for an integer $n$ we denote by $\sC^{\leqslant n}$ and $\sC^{\geqslant n}$  the full subcategories of objects concentrated in degrees at most $n$ and at least $n$, respectively. We denote the heart of the $t$-structure, i.e. the abelian category of objects in cohomological degree zero,  by $\sC^\heartsuit$.

\subsubsection{} For the affine Lie algebra, consider its abelian category $\gk\mod^\heartsuit$ of smooth modules on which $\mathbf{1}$ acts via the identity. We denote by $\gk\mod$ its associated renormalized derived category, as introduced by Frenkel--Gaitsgory \cite{dmod-aff-flag}. We recall that this carries a $t$-structure and a canonical $t$-exact functor to the naive unbounded derived category
\[
    \gk\mod \rightarrow \on{D}(\hspace{.5mm}\gk\mod^\heartsuit),
\]
which is an equivalence on bounded below categories. So, the only difference between the two lies in cohomological degree $-\infty$.

\subsubsection{} \label{sss:wmod}For the affine W-algebra, consider its abelian category $\sW\mod^\heartsuit$ of smooth modules. We denote by $\sW\mod$ its renormalized derived category, as introduced by Raskin \cite{whit}. Again, one has a $t$-exact functor
\[
  \sW\mod \rightarrow \on{D}(\sW\mod^\heartsuit)
\]
which induces an equivalence of the categories of bounded below objects.

\subsection{The plus and minus reductions} The plus reduction below was introduced by Feigin--Frenkel in  \cite{ff90}. The minus reduction was studied for $\fsl_2$ by Feigin--Frenkel \cite{ff902}, and was introduced in general by Frenkel--Kac--Wakimoto \cite{fkw}.

\subsubsection{} From the definition of the affine W-algebra, semi-infinite cohomology defines a functor 
\begin{equation} \label{e:preduc}
C^{\frac{\infty}{2} + *}\big(L\fn^-, L^+\fn^-, (-) \otimes k_\psi^-\hspace{.5mm}\big): \gk\mod \rightarrow \sW\mod. 
\end{equation}
On bounded below objects, the composition 
\[
  \gk\mod^+ \rightarrow \gk\mod \rightarrow \sW\mod \rightarrow \on{D}(\sW\mod^\heartsuit)
\]
agrees with standard construction via Feigin's semi-infinite cochain complex \cite{feigin84}. Following Frenkel--Kac--Wakimoto, we call \eqref{e:preduc} the {plus reduction}. We denote it by $\Psi^+$.

\subsubsection{} \label{sss:defminusred}Fix in addition a generic additive character $\psi$ of the opposite unipotent subalgebra $N$. Associated to our coordinate $t$ is an additive character 
\[
   LN \xrightarrow{L \psi} L \mathbb{G}_a \xrightarrow{\on{Res}(t \cdot -)} \mathbb{G}_a,
\]
where, at the level of $k$-points, $\on{Res}(t \cdot -)$ sends a formal power series to its coefficient of $1$. We denote by $k_\psi$ the associated one dimensional representation of $L\fn$. 

Let us write $L^{(1)}\fn$ for the first congruence subalgebra of $L^+\fn$, so that 
\[
 L^{(1)}\fn \simeq \fn \otimes t \cdot k[\hspace{-.2mm}[t]\hspace{-.2mm}]. 
\]
Consider the associated semi-infinite cohomology.
\begin{equation} \label{e:minusred}
C^{ \frac{\infty}{2} + *}\big(L\fn, L^{(1)}\fn, (-) \otimes k_\psi\big): \gk\mod \rightarrow \on{Vect}. 
\end{equation}

\subsubsection{} The functor \eqref{e:minusred} lifts to a functor to $\sW\mod$ as follows. Recalling that $G$ is adjoint, it follows there is a unique coset representative $\dot{w}_\circ$ for the longest element $w_\circ$ of $W_f$ satisfying
\[
  \on{Ad}_{\dot{w}_\circ}(\psi) = \psi^-, 
\]
i.e. that exchanges the additive characters of $N$ and $N^-$ under conjugation. 

  Write $\crho$ for one half of the sum of the positive coroots of $G$. Passing to the loop groups, it follows that the element $x := \dot{w}_\circ t^{\crho}$ satisfies 
  \[
     \on{Ad}_{x}(\psi) = \psi^-,
  \]
  i.e. it exchanges the additive characters of $LN$ and $LN^-$ under conjugation. 
  
 Using this and the canonical isomorphism 
 \begin{align*}
 C^{\frac{\infty}{2} + *}\big(L\fn, L^{(1)}\fn, (-) \otimes k_\psi\big) & \simeq C^{\frac{\infty}{2} + *}\big(L\fn^-, \on{Ad}_{x} L^{(1)} \fn, \on{Ad}_{x}(-) \otimes k_{\psi^-}\hspace{.5mm}\big)
 \\ & \simeq C^{\frac{\infty}{2} + *}\big(L\fn^-, L^+ \fn^-, \on{Ad}_{x}(-) \otimes k_{\psi^-}\hspace{.5mm}\big) \otimes \ell,
 \intertext{where $\ell$ is the cohomologically graded line} 
   \ell & \simeq \on{rel.det}(\on{Ad}_{x} L^{(1)}\fn, L^+\fn^-), 
 \end{align*}
 we thereby obtain a functor 
 \[
 C^{\frac{\infty}{2} + *}\big(L\fn, L^{(1)}\fn, (-) \otimes k_\psi\big): \gk\mod \rightarrow \sW\mod.
 \]
 Following Frenkel--Kac--Wakimoto, we call this the minus reduction. We  denote it by $\Psi^-$. 

\subsection{Highest weight modules and the minus reduction}

\subsubsection{} Recall that we fixed a Borel subgroup $B$ of $G$, and that we write $I$ for the associated Iwahori subgroup of $LG$. For an element $\lambda \in \ft^*$, we denote the corresponding one dimensional representation of $\ft$ by $k_\lambda$ and the corresponding Verma module of $\gk$ by
\[
      M_\lambda := \on{pind}_{\ft}^{\gk} k_\lambda. 
\]
We denote the simple quotient of $M_\lambda$ by $L_\lambda$. 

\subsubsection{} \label{sss:dotaction}    Recall that the combinatorics of such highest modules are controlled by the level $\kappa$ dot action of $W$ on $\ft^*$, which we denote by 
\[
    W \times \ft^* \rightarrow \ft^*, \quad \quad (w, \mu) \mapsto w\mu. 
\]
Explicitly, $W_f$ acts by the usual dot action of the finite Weyl group, and a coweight $\check{\mu} \in \check{\Lambda}$ acts by translation by 
   \[
      (\kappa_c - \kappa)(\check{\mu}) \in \ft^*, 
   \]
where $\kappa_c$ denotes the critical level. In this formula, we have used the restriction of $(\kappa_c - \kappa)$ from $\fg$ to $\ft$, and the induced map 
\[
       (\kappa_c - \kappa): \ft \rightarrow \ft^*. 
\]

\subsubsection{} \label{sss:intwylgroup}
We also recall that to any $\lambda \in \ft^*$ one associates its integral Weyl group $W_\lambda$. To define this, write $\check{\Phi}^+$ for the  positive real affine coroots, i.e. the positive real coroots of $\gk$. For $\halpha \in \check{\Phi}^+$, consider the dual real root $\alpha$ of $\gk$, and write $\alpha_c$ for its classical part, i.e. associated root of $\fg$. There is unique affine linear functional on $\ft^*$, which we denote again by $\halpha$, such that the corresponding reflection $s_{\halpha}$ in $W$ acts through the level $\kappa$ dot action by  
\[
    s_{\halpha}\mu = \mu - \langle \mu, \halpha \rangle \alpha_c. 
\]

With this, we may associate to $\lambda$ its positive integral coroots 
\[
    \check{\Phi}_\lambda^+ := \{ \halpha \in \check{\Phi}^+: \langle \halpha, \lambda \rangle \in \mathbb{Z} \}.
\]
By definition,  $W_\lambda$ is the subgroup of $W$  generated by the reflections $s_{\halpha},$ for $\halpha \in \check{\Phi}^+_\lambda$. Moreover, if we call an element $\halpha \in \check{\Phi}^+_\lambda$ indecomposable if it cannot be written as a sum of two other elements of $\check{\Phi}^+_\lambda$, then the reflections in $W_\lambda$ associated to the indecomposable integral coroots exhibit $W_\lambda$ as a Coxeter system. Note that any simple reflection of $W$ contained in $W_\lambda$ is necessarily simple in $W_\lambda$, but a  simple reflection of $W_\lambda$ need not be simple in $W$.  

\subsubsection{} \label{ss:hwminus} We next review the associated categories of highest weight modules. Let us write $\Lambda$ for the the character lattice of $T$, i.e. the root lattice. Recall the canonical bijection between character sheaves on $T$ and $\ft^* / \Lambda$, which associates to a character sheaf the monodromy of the associated local system. 

\iffalse
We remark in passing that for $k = \mathbb{C}$ the exponential map canonically identifies $\ft^* / \Lambda$ with the $k$-points of the dual torus $\check{T}$. 
\fi

For a point $\chi \in \ft^* / \Lambda$, we denote by the same letter the associated character sheaf on $I$, and by 
\[
   \gk\mod^{I, \chi}
\]
the category of $(I, \chi)$-equivariant $\gk$-modules. Plainly, this is compactly generated the Verma modules 
\[
     M_\lambda, \quad \text{for } \lambda \in \chi + \Lambda, 
\]
and is a renormalization of the corresponding piece of (cocompleted, derived) Category $\OO$. For more discussion, the reader may see for example \cite{gn}, Section 5 of \cite{tfle}, or Section 9 of  \cite{mys}.

\subsubsection{} We would like to recall the behavior of the minus reduction on $\gk\mod^{I, \chi}$. To do so, let us write $\ft^* /\hspace{-1.1mm}/ W_f$ for the invariant theory quotient of $\ft^*$ by the dot action of $W_f$. We normalize its identification with the spectrum of the Zhu algebra 
\[
   \on{Spec } \on{Zhu}(\sW) \simeq \ft^* /\hspace{-1.1mm}/ W_f 
\]
as in \cite{lpw}. For an infinitesimal character $\Theta \in \ft^* /\hspace{-1.1mm}/ W_f(k)$, we denote the corresponding Verma module for $\sW$ by $M_\Theta$ and its simple quotient by $L_\Theta$.

The following fundamental theorem of Arakawa settled the Frenkel--Kac--Wakimoto conjecture for the minus reduction \cite{ara04}, \cite{ara07}, for an alternative proof cf. \cite{whit}, \cite{locw}. To state it, denote the Harish-Chandra projection by 
\[
  \pi: \ft^* \rightarrow \ft^* /\hspace{-1.1mm}/ W_f,
\]
and recall that a weight $\mu \in \ft^*$ is said to be {finite antidominant} if for every finite positive coroot $\halpha$ one has
\[
    \langle \halpha, \lambda \rangle \notin \mathbb{Z}^{> 0},
\]
where we incorporate a $\rho$-shift into the pairing as in Section \ref{sss:intwylgroup}. 
\begin{theo}\label{t:araminus} The minus reduction  
\[
  \Psi^-: \gk\mod^{I, \chi} \rightarrow \sW\mod
\]
is $t$-exact. Moreover, for any $\lambda \in \chi + \Lambda$ there are canonical isomorphisms 
$
     \Psi^-(M_\lambda) \simeq M_{\pi(\lambda)}, 
     $
and 
\[
 \Psi^-(L_\lambda) \simeq \begin{cases} L_{\pi(\lambda)} & \text{if $\lambda$ is finite antidominant}, \\ 0 & \text{otherwise.} \end{cases}
\]
That is, the minus reduction sends Verma modules to Verma modules and simple modules to simple modules or zero. 
\end{theo}

\subsection{Categorical actions of loop groups}

\label{ss:catacts}

\subsubsection{} We find it convenient and clarifying to use some language and tools from the categorical representation theory of loop groups in what follows. However, the reader may straightforwardly translate our arguments to the setting of \cite{fg06loc} (which, along with \cite{bdh}, laid the groundwork for this theory).

Some references for the ensuing material are as follows. For generalities on dg-categories, one may consult Section 1 of \cite{gaitsroz}. For D-modules on ind-schemes, one may look at \cite{rdm} or the user-friendly \cite{beraldo}. For the basic results in categorical representation theory, one may see  \cite{beraldo}, \cite{mys}, and \cite{whitlocglob}, as well as the notes \cite{paris-notes}.

\subsubsection{} Consider the $(\infty, 2)$-category of cocomplete, presentable $k$-linear dg-categories $\DGCat$. We refer to objects of it as dg-categories, and 1-morphisms in it, i.e. $k$-linear quasi-functors in the sense of \cite{drquo} which commute with colimits, as functors.

$\DGCat$ is a symmetric monoidal $\infty$-category with respect to the Lurie tensor product.\footnote{Using the identification between the arising internal Homs and the morphisms in $\DGCat$, one can alternatively work only with the $(\infty, 1)$-category associated to $\DGCat$, i.e. discard non-invertible 2-morphisms.} The monoidal unit is the category $\on{Vect}$. We denote the underlying binary product of two dg-categories $\sC$ and $\sD$ by $\sC \otimes \sD$, and for objects $c $ of $\sC$ and $d$ of $\sD$ we denote by $c \boxtimes d$ the corresponding `external product' object of $\sC \otimes \sD$. 

Suppose $\sC$ and $\sD$ are compactly generated, which is the case for all dg-categories considered in this paper. Concretely, $\sC \otimes \sD$ is then compactly generated by external products of compact objects from each factor, and the complexes of homomorphisms between such objects are simply the tensor product 
\[
      \on{Hom}_{\sC \otimes \sD}(c_1 \boxtimes d_1, c_2 \boxtimes d_2) \simeq \on{Hom}_{\sC}(c_1, c_2) \otimes \on{Hom}_{\sD}(d_1, d_2) 
\]
with the natural compositions.

The category of $\kappa$-twisted D-modules on the loop group, which we denote by $D_\kappa(LG)$, is naturally an algebra object of $\DGCat$. The underlying binary product is given by convolution, i.e. pushforward along the multiplication map$$LG \times LG \rightarrow LG.$$
We denote its $(\infty,2)$-category of left modules by $D_\kappa(LG)\mod$. Plainly, an object of $D_\kappa(LG)\mod$ is a dg-category $\sC$ equipped with an action 
\[
     D_\kappa(LG) \otimes \sC \rightarrow \sC 
\]
which is unital and associative up to coherent homotopy, and a 1-morphism between representations is a functor between the underlying dg-categories equipped with a datum of $D_\kappa(LG)$-equivariance.

The main example of interest for us is the canonical $D_\kappa(LG)$-module structure on  $\gk\mod$, see for example \cite{mys}. To get a feel for this action, we note that for an element $g \in LG$, convolution with the corresponding delta function 
\[
    \delta_g \star - : \gk\mod \simeq \gk\mod
\]
is induced by the adjoint action of $g$ on $\gk$.

\subsection{Invariant vectors and representations} Given a $D_\kappa(LG)$-module $\sC$, one can form its categories of twisted Iwahori and Whittaker invariant objects 
\[
    \sC^{I, \chi} \quad \text{and} \quad \sC^{LN, \psi}. 
\]
In our example of interest, these will exactly be the highest weight modules $\gk\mod^{I, \chi}$ of Section \ref{ss:hwminus}, and the category $\sW\mod$ of Section \ref{sss:wmod}, respectively. We now review this in more detail. 

\subsubsection{} We begin with the general construction of invariants. So, consider an ind-affine placid group ind-scheme $H$ equipped with a character sheaf $\phi$, see e.g. \cite{rdm} for a discussion of placidity.  

\begin{re} The following technical point may be safely ignored by the reader. The category of D-modules on a placid ind-scheme $X$ which we denote by $D(X)$ is denoted in \cite{rdm} and \cite{beraldo} by $D^*(X)$, and comes with functoriality for $*$-pushforwards. There is a  dual category $D^!(X)$ which comes with functoriality for $!$-pullbacks. The latter is symmetric monoidal under $!$-tensor product, and canonically acts on $D^*(X)$ by $!$-tensor product. In particular, when we speak of a character sheaf $\phi$, it canonically is an object of $D^!(X)$. However, a choice of a dimension theory on $X$, which always exists in our cases of interest, identifies $D^!(X)$ and $D(X)$.  \label{r:tensor} \end{re}

Associated to the character sheaf $\phi$ is a $D(H)$-module $\on{Vect}_\phi$. Explicitly, its underlying dg-category is $\on{Vect}$, and the action map is of the form 
\begin{equation} \label{e:actmap}
   D(H) \otimes \on{Vect}_\phi \rightarrow \on{Vect}_\phi, \quad \quad M \boxtimes V \mapsto \Gamma_{\on{dR}}(H, M \overset ! \otimes \phi ) \otimes V,
\end{equation}
where in the right-hand expression $\overset ! \otimes$ is the $!$-tensor product $D(X) \otimes D^!(X) \rightarrow D(X)$, cf. Remark \ref{r:tensor}, and the outer tensor product is of $k$-vector spaces. For any $D(H)$-module $\sC$, one may accordingly form its invariants and coinvariants 
\[
       \sC^{H, \phi} := \on{Hom}_{D(H)\mod}(\on{Vect}_\phi, \sC) \quad \text{and} \quad \sC_{H, \phi} := \on{Vect}_\phi \underset{D(H)} \otimes \sC. 
\]
Recalling the tautological equivalence $$D(H) \otimes \on{Vect}_\phi \simeq D(H),$$the action map \eqref{e:actmap} underlies a $D(H)$-equivariant functor $D(H) \rightarrow \on{Vect}_\phi$, which induces canonical `forgetting' and `inserting' maps 
\[
      \on{Oblv}: \sC^{H, \phi} \rightarrow \sC \quad \text{and} \quad \on{ins}: \sC \rightarrow \sC_{H, \phi}.  
\]
\label{s:placid}
\subsubsection{}\label{sss:compactopen} Suppose further now that $H$ is a group scheme. In this case, the map $D(H) \rightarrow \on{Vect}_\phi$ admits a left adjoint (i.e., $D(H)$ contains a constant sheaf) and hence one obtains an `averaging' adjunction 
\[
     \on{Oblv}: \sC^{H, \phi} \leftrightarrows \sC: \on{Av}_*^{H, \phi}. 
\]
Moreover, $\on{Av}_*^{H, \phi}$ canonically factors through through $\on{ins}$, and if the prounipotent radical of $H$ is of finite codimension the resulting map is an equivalence 
\[
   \sC_{H, \phi} \xrightarrow{\sim} \sC^{H, \phi}. 
\]
We freely pass between invariants and coinvariants for such groups in what follows.

Let us take $H$ to be the Iwahori $I$. The twisting $\kappa$ is canonically trivialized on $L^+G$, and we use the induced trivialization on its subgroup  $I$ going forwards. In particular, any $D_\kappa(LG)$-module $\sC$ carries an action of $D(I)$. With this, for an element $\chi \in \ft^* / \Lambda$ we may form the twisted invariants $\sC^{I, \chi}$. Specializing to $$\sC = \gk\mod,$$the obtained category $\gk\mod^{I, \chi}$ canonically identifies with the renormalized derived category of Harish-Chandra modules discussed in Section \ref{ss:hwminus}, cf. the references therein for the details. We note in passing that a similar statement holds for any compact open subgroup $H$ of $LG$.

\subsubsection{} \label{sss:whit} The contents of Section \ref{sss:compactopen} admit an analogue for the Whittaker invariants, despite $LN$ being an ind-scheme. Recall the generic additive character
\[
   \psi: LN \rightarrow \mathbb{G}_a
\]
of Section \ref{sss:defminusred}. We denote by the same letter the character sheaf $\psi^!( \on{exp}(z))$, where $\on{exp}(z)$ is the character sheaf on $\mathbb{G}_a$ whose underlying D-module is the exponential D-module placed in cohomological degree -1. 

The twisting $\kappa$ is canonically trivial on $LN$, and so to a $D_\kappa(LG)$-module $\sC$ one may form its Whittaker (co)invariants. It will be convenient for us to use the canonical equivalence 
\[
   \sC_{LN, \psi} \simeq  \sC^{LN, \psi},
\]
conjectured by Gaitsgory, and proven by Beraldo for $GL_n$ and in general by Raskin \cite{beraldo}, \cite{whit}. 

We denote by $\on{Av}_{*}^{\psi, r}$ the renormalized averaging functor, i.e. the composition 
\[
    \sC \xrightarrow{\on{ins}} \sC_{LN, \psi} \simeq \sC^{LN, \psi}. 
\]
It is given by convolution with an object, which we call the (twisted) renormalized constant sheaf
\[
       k^{\psi, r} \in D(LN)^{LN, \psi} \subset D_\kappa(LG)^{LN, \psi}. 
\]

\subsubsection{} \label{s:normcon}Let us recall an explicit presentation of $k^{\psi, r}$. To do so, first consider any ind-scheme of the form 
\[
 A \simeq \varinjlim_\alpha A_\alpha, 
\]
where each $A_\alpha$ is a pro-finite dimensional affine space and each transition map $A_\alpha \hookrightarrow A_\beta$ is the closed embedding of an open affine subspace of finite codimension (here, the word open means with respect to the pro-topology on $A_\beta$). Each $A_\alpha$, as with any placid scheme, carries a canonical constant sheaf $k(A_\alpha)$. Namely, consider any presentation of $A_\alpha$ as a cofiltered limit
\[
  A_\alpha \simeq \varprojlim_\gamma    A_\alpha^\gamma,
\]
where the $A_\alpha^\gamma$ are finite dimensional affine spaces, and the transition maps are smooth surjections. It follows that the $*$-pushforwards $D(A_\alpha^\gamma) \rightarrow D(A_\alpha^{\gamma'})$ admit left adjoints, i.e. $*$-pullbacks. We have by definition that
\[
    D(A_\alpha) \simeq \varinjlim_\gamma D(A_\alpha^\gamma),
\]
i.e. its category of D-modules is the colimit in $\DGCat$ of those on its finite dimensional quotients under $*$-pullback. In particular, the insertion of the constant sheaf from any step in the colimit is the sought-for constant sheaf $k(A_\alpha)$.

For any fixed index $\alpha_\circ$, there is an associated renormalized constant sheaf $k(A, A_{\alpha_\circ})$ on all of $A$. 
To describe this, note that by our assumptions the $*$-pushforwards $D(A_\alpha) \rightarrow D(A_\beta)$ admit right adjoints, i.e. $!$-pullbacks. With this, $k(A, A_{\alpha_\circ})$ is determined uniquely up to unique isomorphism by the properties that its $!$-restriction to any $A_\beta$ is isomorphic to a cohomological shift of a constant sheaf, and its $!$-restriction to $A_{\alpha_\circ}$ is $k(A_{\alpha_\circ})$.  Explicitly, we have by definition 
\[
 D(A) \simeq \varprojlim_\alpha D(A_\alpha),
\]
i.e. its category of D-modules is the cofiltered inverse limit in $\DGCat$ of those on open profinite dimensional subspaces under $!$-pullback. Hence, the existence follows from the fact that such $!$-pullbacks send, by our assumption on the maps $A_\alpha \hookrightarrow A_\beta$, constants sheaves to shifts of constant sheaves. 

\begin{re} As the notation suggests, the objects $k(A, A_{\alpha_\circ})$ are semi-infinite in nature, appearing like a constant sheaf along the pro-directions and a dualizing sheaf along the ind-directions. Informally, they may be thought of as halfway between the dualizing sheaf and the constant sheaf, which do not exist as objects of $D(A)$. In particular, the various $k(A, A_{\alpha})$ agree up to tensoring by cohomologically graded lines.  
\end{re}

Let us apply this discussion to $A = LN$. We then have, in the conventions of the present paper, that the kernel for renormalized averaging is 
\begin{equation} \label{e:avr}
     k^{\psi, r} \simeq k(LN, L^+N) \overset ! \otimes \psi. 
\end{equation}

\subsubsection{}\label{sss:affskr} It will also be convenient to use the affine Skryabin equivalence 
\[
       \gk\mod^{LN, \psi} \simeq \sW\mod,
\]
which is an important theorem of Raskin \cite{whit}. We will normalize this by asking that it exchange the functor $\on{Av}_*^{\psi, r}$ and the semi-infinite cohomology
\[
     C^{\frac{\infty}{2} + *}\big(L\fn, L^+\fn, (-) \otimes k_\psi \big): \gk\mod \rightarrow \sW\mod. 
\]

\begin{re} In the alternative proof of Arakawa's Theorem \ref{t:araminus} by the author and Raskin, the affine Skryabin equivalence plays a central role. Here, modulo the facts about the minus reduction, we will only use that semi-infinite cohomology factors through the Whittaker coinvariants, which is more or less automatic. \end{re}

\subsection{Convolution algebras}

\subsubsection{} Suppose that $H$ is as in Section \ref{s:placid}, and is equipped with an embedding $H \hookrightarrow LG$ and a trivialization $\tau$ of the restriction of $\kappa$. In this case, by tensor-hom adjunction we have for any $D_\kappa(LG)$-module $\sC$ an equivalence 
\begin{equation} \label{e:mapequiv}
\sC^{H, \phi} \simeq \on{Hom}_{D_\kappa(LG)\mod}(D_\kappa(LG)_{H, \phi}, \sC),
\end{equation}
where the appearing $(H, \phi)$-coinvariants are taken with respect to the right $D(H)$ action induced by right multiplication. 

\begin{re} If we were to instead use the left $D(H)$ action induced by right multiplication, we would rather consider $D_\kappa(LG)_{H, \phi^\vee}$, where $\phi^\vee$ is the inverse character sheaf. For ease of notation, we follow the former convention. \end{re}

The equivalence \eqref{e:mapequiv} yields a $D_\kappa(LG)$-equivariant evaluation map 
\[
      D_\kappa(LG)_{H, \phi} \otimes \sC^{H, \phi} \rightarrow \sC, \quad \quad M \boxtimes c \mapsto M \overset H \star c. 
\]
In particular, for a pair of such subgroups $(H_i, \phi_i, \tau_i)$, $i = 1, 2,$  one has a canonical map 
\[
     D_\kappa(LG)_{H_2, \phi_2}^{H_1, \phi_1} \otimes \sC^{H_2, \phi_2} \rightarrow \sC^{H_1, \phi_1}. 
\]
If both $H_1$ and $H_2$ are as in Sections \ref{sss:compactopen} and \ref{sss:whit}, so that we may canonically identify their invariants and coinvariants, we write this as 
\[
        D_\kappa(H_1, \phi_1 \backslash LG / H_2, \phi_2) \otimes \sC^{H_2, \phi_2} \rightarrow \sC^{H_1, \phi_1}. 
\]
In the remainder of this section, we collect some basic notation for the specific categories of bi-equivariant sheaves which will appear later.

\subsection{Twisted affine Hecke categories} \label{ss:twaffheck}
Fix a character sheaf $\chi \in \ft^* / \Lambda$, write $\mathring{I}$ for the prounipotent radical of $I$, and consider the category 
\[
      D_\kappa(\mathring{I} \backslash LG / I, \chi). 
\]

Recall the Bruhat decomposition, i.e. that $LG$ admits a  stratification with strata 
$
 \mathring{I} w I,$ for  $w \in W$. It is standard that the bi-equivariant sheaves on a fixed stratum 
$
       D_\kappa(\mathring{I} \backslash \mathring{I}wI / I, \chi)
$
admit an equivalence with $\on{Vect}$. Indeed, for any coset representative $\dot{w}$ of $w$, consider the projection and multiplication maps
\[
       I \xleftarrow{\pi} \mathring{I} \times \dot{w} \times I \xrightarrow{\mu}  \mathring{I}wI.
\]
We may write $\pi^! \chi$ as $\mu^! \phi$, for some invertible sheaf $\phi$ on $\mathring{I}wI$. With this, the functor from $\on{Vect}$ sending $k$ to the twisted constant sheaf $$k(\mathring{I}wI) \overset ! \otimes \phi$$is an equivalence. This is only canonical up to tensoring by a line in cohomological degree zero, namely a $!$-fibre of $\chi$. 

We denote the corresponding standard and costandard objects, i.e. the $!-$ and $*$-extensions of $k[\ell(w)]$ in $\on{Vect}$, by 
\[
     j_{w, !}^{w \chi, \chi} \text{ and } j_{w, *}^{w\chi, \chi} \in D_\kappa(\mathring{I} \backslash LG / I, \chi),
\]
respectively. In particular, we will only define these objects up to tensoring by a line in cohomological degree zero, which suffices for our purposes. We also follow the practice of calling the objects $j_{w, *}^{w \chi, \chi}$ intertwining functors, in view of their function-theoretic counterparts.

  Both $j_{w, !}^{w \chi, \chi}$ and $j_{w, *}^{w \chi, \chi}$  descend to $D_\kappa(I, w\chi \backslash LG / I, \chi)$, where we use the level $\kappa$ action of $W$ on $\ft^* / \Lambda$ to obtain $w\chi$ from $\chi$. Plainly, this action may be constructed by noting that the level $\kappa$ dot action of $W$ on $\ft^*$, cf. Section \ref{sss:dotaction}, descends to a well-defined action on 
$
     \ft^* / \Lambda. 
$

Fix any $\lambda \in \ft^*$ and real affine coroot $\halpha$ of $\gk$. As $G$ is adjoint, and in particular $\Lambda$ is the root lattice, it follows that
\[
       s_{\halpha} \in W_\lambda \quad \text{if and only if} \quad  s_{\halpha} (\lambda + \Lambda) = \lambda + \Lambda,
\]
cf. Section \ref{sss:intwylgroup} for a discussion of the integral Weyl group $W_\lambda$. In particular, we may speak of the integral Weyl group $W_\chi$ of a twist $\chi$. Namely, it is the subgroup of $W$ generated by the $s_{\halpha}$ fixing $\chi$, or equivalently the integral Weyl group of any lift $\tilde{\chi} \in \ft^*$. 
\label{ss:twheckecat}

\subsection{Twisted Whittaker sheaves on the affine flag manifold} \label{ss:twwhitaff} Finally, we would like to discuss some basic properties of the category 
\[
   D_\kappa(LN, \psi \backslash LG / I, \chi). 
\]

Recall by the Iwasawa decomposition that $LG$ admits a stratification with strata 
$LNwI,$ for $w \in W$. Denote by $\Wf$ the subset of $W$ given by elements of minimal length in their right $W_f$ cosets. The categories on each stratum admit equivalences 
\[
    D_\kappa( LN, \psi \backslash LNwI / I, \chi) \simeq \begin{cases} \on{Vect} & \text{if $w \in \Wf$}, \\ 0 & \text{otherwise}. \end{cases}
\]
To describe the former equivalences, again fix a coset representative $\dot{w}$ and consider the multiplication 
\[
    LN \times I \simeq LN \times \dot{w} \times I \xrightarrow{\mu} LNwI.
\]
By the assumption that $w \in \Wf$, we may write $\psi \boxtimes \chi$ as $\mu^! \phi$, for an invertible sheaf $\phi$ on $LNwI$. The map from $\on{Vect}$ sending $k$ to the twisted constant sheaf 
\begin{equation} \label{e:twcs}
       k^{\psi, \chi} :=  k(LNwI, L^+NwI) \overset ! \otimes \phi
\end{equation}
is an equivalence, where the renormalized constant sheaf $k(LNwI, L^+NwI)$ is defined as in Section \ref{s:normcon}. We denote the corresponding costandard object, i.e. the $*$-extension of $k^{\psi, \chi}[\ell(w_\circ w)]$, by 
\begin{equation} \label{e:costwhit}
    j_{w, *}^{\psi, \chi}.
\end{equation}

\section{Spectral flow and intertwining functors}
\label{s:sfint}

   The contents of this section are as follows. In the first half, we consider the conjugations of Drinfeld--Sokolov reduction by arbitrary elements of $LG$. When restricted to the category of highest weight modules, we find that the resulting (nonzero) functors are in canonical bijection with the double cosets in 
   \[
      LN \backslash LG / I
   \]
   which support Whittaker sheaves. Concretely, these are parametrized by the coweight lattice, and the corresponding functors are essentially the usual  spectrally flowed reduction functors, see Remark \ref{r:spectralflow} below for the precise relation.

   In the second half, we provide a canonical factorization of each spectrally flowed reduction applied to a block of highest weight modules. These depend on the block  
   and refine the fact, which essentially appears in Frenkel--Gaitsgory \cite{fg06loc}, that the various reductions all differ from one another by intertwining functors. As we will see in Section \ref{s:fkw}, these factorizations underlie the combinatorics of the Frenkel--Kac--Wakimoto conjecture for the plus reduction.

\subsection{} For a generic character $\psi$ of $LN$ of conductor one, cf. Section \ref{sss:defminusred}, consider the semi-infinite cohomology functor 
\[
  \mathring{\Psi}^e := C^{\frac{\infty}{2} + *}( L\fn, L^+\fn, - \otimes k_\psi): \gk\mod \rightarrow \sW\mod.  
\]
With our intended applications to the plus reduction in mind, we will consider in this section the conjugates of $\mathring{\Psi}^e$ by elements of $LG$, and study their restriction to highest weight modules.

\begin{re} \label{r:mred}As the minus reduction $\Psi^-$ is normalized instead with respect to the first congruence subalgebra of $L^+\fn$, by trivializing its relative determinant with $L^+\fn$, cf. Section \ref{sss:reldet}, one obtains an equivalence
\[
   \mathring{\Psi}^e[ \ell(w_\circ)] \simeq \Psi^-. 
\]
\end{re}

\subsection{} Fix any element $g \in LG$. Conjugation by $g$ yields an automorphism $\on{Ad}_g: \gk \rightarrow \gk.$ We denote the corresponding pushforward of modules, i.e. restriction along $\on{Ad}_{g^{-1}}$, by 
\[
   \on{Ad}_{g, *}: \gk\mod \rightarrow \gk\mod. 
\]

We denote by $\mathring{\Psi}^g$ its composition with the previous semi-infinite cohomology functor, i.e.   
\begin{equation} \label{e:gdsred}
      \mathring{\Psi}^g : \gk\mod \xrightarrow{\on{Ad}_{g, *}} \gk\mod \xrightarrow{\mathring{\Psi}^e} \sW\mod.
\end{equation}

\begin{ex} \label{e:pred}Fix a coset representative $\dot{w}_\circ$ in $N_G(T)$ for the longest element $w_\circ$ of $W_f$. Let us specialize the preceding discussion to $g = \dot{w}_\circ t^{\rhoc}$. In this case \eqref{e:gdsred} is, up to a cohomological shift coming from the normalizing compact open subalgebras, the plus Drinfeld--Sokolov reduction. 

More precisely, as the plus reduction is normalized with respect to $L^+\fn^-$, rather than $$\on{Ad}_{t^{-\crho}} L^+\fn^-,$$by trivializing the appearing relative determinant line we obtain an isomorphism  
\[
  \mathring{\Psi}^{\dot{w}_\circ t^{\check{\rho}}}[2\langle \rho,  \crho \rangle] \simeq \Psi^+. \
\]
\end{ex}

\subsection{} We will be interested in the behavior of the functor  \eqref{e:gdsred} on highest weight modules. More precisely, let us fix $\chi \in \ft^* / \Lambda$, which we identify with the corresponding character sheaf on $I$, and consider the functor 
\begin{equation} \label{e:gminusred}
   \gk\mod^{I, \chi} \xrightarrow{\on{Oblv}} \gk\mod  \xrightarrow{\on{Ad}_{g, *}} \gk\mod \xrightarrow{\mathring{\Psi}^e} \sW\mod,
\end{equation}
cf. Section \ref{ss:hwminus} for a discussion of $\gk\mod^{I, \chi}$. We will rewrite \eqref{e:gminusred} as follows. For any $D_\kappa(LG)$-module module $\sC$, we may form the composition 
\begin{equation} \label{e:catcomp}
    \sC^{I, \chi} \xrightarrow{\on{Oblv}} \sC \xrightarrow{\delta_g \star - } \sC \xrightarrow{\Avr} \sC^{LN, \psi};
\end{equation}
the reader may consult Section \ref{ss:catacts} for a discussion of the appearing terms. Taking $\sC = \gk\mod$ then recovers \eqref{e:gminusred}. In particular, \eqref{e:gminusred} is given by convolution with an object 
\[
   \mathring{\mathscr{K}}^g \in D_\kappa(LN, \psi \backslash LG / I, \chi). 
\]

To identify it explicitly, recall that on a stratum $LNgI$ admitting nonzero Whittaker sheaves we denote by $k^{\psi, \chi}$ the $(\psi, \chi)$-twist of its renormalized constant sheaf, cf. Equation \eqref{e:twcs}. 

\begin{pro} \label{l:kg} Consider the double coset $LNgI$. If this does not support Whittaker sheaves, i.e. 
\[
      D(LN, \psi \backslash LNgI / I, \chi) \simeq 0,
\]
then $\mathring{\mathscr{K}}^g$ vanishes. Otherwise, $\mathring{\mathscr{K}}^g$ is the $*$-extension of the renormalized constant sheaf $k^{\psi, \chi}$ on $LNgI$.
\end{pro}

\begin{proof} It follows from \eqref{e:catcomp} that $\mathring{\mathscr{K}}^g$ is the $*$-pushforward along the multiplication map 
\[
       LN \times I \simeq LN \times \{g\} \times  I \xrightarrow{\mu} LNgI \hookrightarrow LG
\]
of the sheaf $k^{\psi, r} \boxtimes \chi$, where $k^{\psi, r}$ is the kernel for the renormalized Whittaker averaging functor $\on{Av}_*^{\psi, r}$, cf. Sections \ref{sss:whit} and \ref{s:normcon} for further discussion. As $\mu$ is a fibration by prounipotent groups and is compatible with the trivializations of dimension torsors used in forming $k^{\psi, r}$ and $k^{\psi, \chi}$, cf. Equations \eqref{e:avr} and \eqref{e:twcs} respectively, the lemma follows.
\end{proof}

\subsection{} \label{ss:onlyondoublecoset}By Proposition \ref{l:kg}, the functor \eqref{e:gminusred} depends only on the double coset $LNgI$. In particular, if $\mathring{\mathscr{K}}^g$ is nonvanishing, we may write 
\begin{equation} \label{e:doublecoset}
    LNgI = LNwI \quad \quad \text{for some } w \in \Wf, 
\end{equation}
where as in Section \ref{ss:twwhitaff} we denote by $\Wf$ the elements of the extended affine Weyl group of minimal length in their right $W_f$ cosets. 

Therefore, if we introduce the cohomological shifts
 $
    \Psi^g := \mathring{\Psi}^g[\ell(w_\circ w)]
 $
and $\mathscr{K}_g := \mathring{\mathscr{K}}^g[\ell(w_\circ w)]$,
we may  deduce the following from Proposition \ref{l:kg}. 

\begin{cor} \label{c:shiftds} For $g$ and $w$ as in Equation \eqref{e:doublecoset}, there exists an isomorphism 
\begin{equation} \label{e:sintop}
    \mathscr{K}^g \simeq j_{{w}, *}^{\psi, \chi},
\end{equation} 
where $j_{w, *}^{\psi, \chi}$ is as in Equation \eqref{e:costwhit}.
\end{cor}
In what follows, we refer to such functors as the spectrally flowed Drinfeld--Sokolov reductions.

\begin{ex}\label{e:pmareint} As important special cases, Corollary \ref{c:shiftds} identifies the plus and minus reductions with the corresponding costandard objects. Namely, for the minus reduction, it follows from Remark \ref{r:mred} and Corollary \ref{c:shiftds} that one has isomorphisms
\[
    \Psi^- \simeq \Psi^e \quad \text{and} \quad \mathscr{K}^e \simeq  j_{e, *}^{\psi, \chi} . 
\]
For the plus reduction, note that the element $w_\circ e^{\crho} \in W$ lies in $\Wf$, and has length 
$$ 2 \langle \rho, \crho \rangle - \ell(w_\circ).$$
It therefore follows from Example \ref{e:pred} and Corollary \ref{c:shiftds} that one has isomorphisms
\[
    \Psi^+ \simeq \Psi^{\dot{w}_\circ t^{\crho}}\quad \text{and} \quad \mathscr{K}^{\dot{w}_\circ t^{\crho}} \simeq j_{w_\circ e^{\crho}, *}^{\psi, \chi}. 
\]
\end{ex}

\begin{re} \label{r:spectralflow}We should mention that our usage of the term spectrally flowed reduction is slightly nonstandard. To explain this, consider the tautological identifications  
\[
     \Wf \hspace{.1cm} \simeq\hspace{.1cm} W_f \backslash W \hspace{.1cm}\simeq\hspace{.1cm} \check{\Lambda}.
\]
For any $\check{\mu} \in \check{\Lambda}$ let us write $w e^{\check{\mu}}$ for the corresponding element of $\Wf$, i.e. the minimal length element in its right $W_f$ coset. By definition, our associated reduction functor $$\Psi^{\check{\mu}} := \Psi^{\dot{w} t^{\check{\mu}}}$$ is given up to an overall cohomological shift by (i) first applying the spectral flow $\on{Ad}_{t^{\check{\mu}}}$ to a $\gk$-module, (ii) further conjugating by $\on{Ad}_{\dot{w}_f}$, and then (iii) applying the minus reduction. We note that if $\check{\mu}$ is  antidominant, step (ii) is trivial, in which case this agrees up to notational conventions with the standard usage, see e.g. \cite{af19}.  
\end{re}

\subsection{} In the remainder of this section we would like to record a canonical factorization, for any $w \in \Wf$, of the corresponding costandard object
\[
    j_{{w}, *}^{\psi, \chi}. 
\]
%
\iffalse
As we will see in the next section, this factorization will exactly reproduce the combinatorics of the Frenkel--Kac--Wakimoto conjecture. 
\fi

We begin with the factorization at the level of the affine Weyl group. Consider the integral Weyl group $W_\chi$ of the character sheaf $\chi$, cf. the final paragraphs of Section \ref{ss:twheckecat}. Recall this is a Coxeter system, i.e. a Coxeter group with a preferred set of generators. 

\begin{lemma}For any $w \in W$, there is a unique element $w_-$ of minimal length in the double coset 
\begin{equation}
    W_f w W_\chi.
\end{equation}
Moreover, the $w_-$-conjugated intersection  \begin{equation} \label{e:parsb} w_-^{-1} W_f w_- \cap W_\chi\end{equation}
is a parabolic subgroup of $W_\chi$. In particular, if the  subgroup \eqref{e:parsb} is trivial, $w$ admits a unique factorization 
\begin{equation} \label{e:canfact}
   w = w_f   w_- w_\chi, \quad \quad \quad \text{where  $w_f \in W_f$ and  $w_\chi \in W_\chi$.} 
\end{equation}
\label{l:coxmin}
\end{lemma}

\begin{re} To orient the reader, the relation of the statement of the lemma to Frenkel--Kac--Wakimoto's conjecture will be as follows. The group appearing in Equation \eqref{e:parsb} will be nontrivial if and only if the reduction vanishes. When it is trivial, the canonical factorization of Equation \eqref{e:canfact}, and in particular the length of $w_\chi$ in the Coxeter system $W_\chi$, will produce the cohomological shift of the conjecture. 
\end{re}

\begin{proof}[Proof of Lemma \ref{l:coxmin}] We proved the first two assertions in Lemma 3.2.9 of \cite{qfle}. The final statement follows by noting that the stabilizer of $w_-$ under the right action 
\[
    W \times W_f \times W_\chi \rightarrow W, \quad \quad y \cdot (w_f, w_\chi) = w_f^{-1} y w_\chi,
\]
is identified via the projection onto $W_\chi$ with \eqref{e:parsb}. 
. \end{proof}

We next lift Lemma \ref{l:coxmin} to a statement about costandard objects as follows. Fix an element $w \in \Wf$, and consider the projection 
\[
 \pi: W \rightarrow W_f \backslash W.
\]
By Lemma \ref{l:coxmin}, we have a canonical identification of the arising $W_\chi$-orbit as  
\[
     \pi(w) \cdot W_\chi \simeq (w_-^{-1} W_f w_- \cap W_\chi) \backslash W_\chi, 
\]
which is afforded by acting on the minimal length element $\pi(w_-)$. As the stabilizer of the latter is a parabolic subgroup of $W_\chi$, we may canonically write 
\begin{equation} \label{e:canfact2} 
      \pi(w) = \pi(w_-) \cdot  w_\chi, 
\end{equation} 
where $w_\chi \in W_\chi$ is the unique element of minimal length in its right $w_-^{-1} W_f w_- \cap W_\chi$ coset. With these preparations, we may describe the lift to intertwining functors.

\begin{pro} Let $w$, $w_-, w_\chi$ be as in the preceding paragraph. There is an isomorphism 
\[
    j_{{w}, *}^{\psi, \chi} \hspace{1mm}\simeq  \hspace{1mm}j_{e, *}^{\psi, w_- \chi}  \hspace{1mm}\overset I \star  \hspace{1mm}j_{{w}_-, *}^{w_- \chi, \chi} \hspace{1mm} \overset I \star  \hspace{1mm} j_{{w}_\chi, *}^{\chi, \chi},
\]
and the appearing object $j_{{w}_-, *}^{w_-\chi, \chi}$ is clean, i.e. coincides with the standard object $j_{{w}_-, !}^{w_- \chi, \chi}$. \label{p:3steps}
\end{pro}

\begin{proof} We will obtain the desired factorization in two steps. First, we exhibited an isomorphism 
\begin{equation} \label{e:fact1}
    j_{{w}, *}^{\psi, \chi}\hspace{1mm} \simeq\hspace{1mm} j_{{w}_-, *}^{\psi, \chi} \hspace{1mm}\overset I \star \hspace{1mm}j_{{w}_\chi, *}^{\chi, \chi}
\end{equation}
in Proposition 3.2.12 of \cite{qfle}.\footnote{To aid the reader consulting {\em loc. cit.}, we note that the subgroups $N$ and $N^-$ of the present paper are denoted there by $N^-$ and $N$, respectively.} Second, we may further factor the right-hand side via an isomorphism 
\[
     j_{{w}_-, *}^{\psi, \chi} \hspace{1mm}\simeq  \hspace{1mm}j_{e, *}^{\psi, w_- \chi}  \hspace{1mm}\overset I \star  \hspace{1mm}j_{{w}_-, *}^{w_- \chi, \chi},
\]
which along with the cleanness of $j_{{w}_-, *}^{w_- \chi, \chi}$ we recorded in {\em loc. cit.} as Proposition 3.3.8. \end{proof}

\subsection{} Explicitly, if we combine Remark \ref{r:mred}, Corollary \ref{c:shiftds} and  Proposition \ref{p:3steps}, we obtain the following.

\begin{cor} \label{c:3steps}Fix an element $w \in \Wf$ along with a coset representative $\dot{w}$. Then, for any character $\chi$, the spectrally flowed Drinfeld--Sokolov reduction 
\[
    \Psi^{\dot{w}}: \gk\mod^{I, \chi} \rightarrow \sW\mod
\]
is naturally isomorphic to the composition
\[
      \gk\mod^{I, \chi} \xrightarrow{ j_{{w}_\chi, *}^{\chi, \chi} } \gk\mod^{I, \chi} \xrightarrow{ j_{{w}_-, *}^{w_-\chi, \chi}} \gk\mod^{I, w_-\chi} \xrightarrow{\Psi^-} \sW\mod,
\]
where $w_\chi$ and $w_-$ are as in Proposition \ref{p:3steps}. That is, the spectrally flowed reduction is a composition of intertwining functors and the minus reduction.
\end{cor}

In particular, taking $w = w_\circ e^{\crho}$, the above rewrites the plus reduction via intertwining functors and the minus reduction, cf. Example \ref{e:pmareint}. Let us now apply this to our motivating problem.  

\iffalse 
\begin{ex}
Explicitly for the plus reduction we obtain the following. First note that the element $w_\circ t^{\crho}$ lies in $\Wf$, and has length 
%
$$ 2 \langle \rho, \crho \rangle - \ell(w_\circ).$$
%
Accordingly, by combining Corollary \ref{c:3steps} with Example \ref{e:pred}, we obtain that 
%
\[
     \Psi^+[\ell(w_\circ)]: \gk\mod^{I, \chi} \rightarrow \sW\mod
\]
%
is naturally isomorphic to the composition 
%
\[
 \gk\mod^{I, \chi} \xrightarrow{ j_{{w}_\chi, *}^{\chi, \chi} } \gk\mod^{I, \chi} \xrightarrow{ j_{{w}_-, *}^{w_-\chi, \chi}} \gk\mod^{I, w_-\chi} \xrightarrow{\Psi^-} \sW\mod,
\]
%
where $w_\chi$ and $w_-$ are the elements associated to $w = w_\circ t^{\crho}$ as in Proposition \ref{p:3steps}.  
\end{ex}
\fi

\section{The Frenkel--Kac--Wakimoto conjecture}
\label{s:fkw}

In this section, we calculate the spectrally flowed reductions of certain irreducible highest weight modules. The structure is as follows. Bearing in mind the factorization of Corollary \ref{c:3steps}, we first establish some preliminary assertions about the behavior of the appearing intertwining functors. With these in hand, the main result is obtained in Theorem \ref{t:maintheorem5}.

\subsection{} Let us begin by establishing notation. Suppose that $\lambda \in \ft^*$ is a dominant weight at level $\kappa$. That is, we assume the Verma module $M_\lambda$ is projective in the Category $\OO$ of $\gk$-modules. By the theorem of Kac--Kazhdan on singular vectors \cite{kk}, this is equivalent to the condition that $\kappa$ is noncritical and for every positive real affine coroot $\halpha$ one has 
\[
    \langle\hspace{.3mm} \halpha, \lambda \rangle \notin \mathbb{Z}^{< 0},
\]
cf. Section \ref{sss:intwylgroup} for a discussion of the appearing pairing. 

 We would like to calculate the spectrally flowed Drinfeld--Sokolov reductions of the corresponding irreducible module
\[
     \Psi^{\dot{w}} (L_\lambda), \quad  \quad \text{for } w \in \Wf, 
\]
where $\Psi^{\dot{w}}$ is as in Section \ref{ss:onlyondoublecoset}.

\subsection{} To do so, write $\chi$ for the character sheaf on $I$ associated to the root lattice coset $\lambda + \Lambda \in \ft^* / \Lambda$, so that 
\[
    L_\lambda \in \gk\mod^{I, \chi}. 
\]
Recall from Corollary \ref{c:3steps} the associated factorization 
\begin{equation} \label{e:e3steps}
   \Psi^{\dot{w}} \simeq \Psi^- \circ j_{{w}_-, *}^{w_- \chi, \chi} \circ j_{{w}_\chi, *}^{\chi, \chi}.
\end{equation}
We now collect some facts about the behavior of the first steps, i.e. the intertwining functors.

\subsection{} It will be convenient to begin with the middle factor $j_{{w}_-, *}^{w_- \chi, \chi}$. Recalling it is clean, we will control it as follows. 

\begin{pro}\label{p:cleanconv} Fix a character sheaf $\phi \in \ft^* / \Lambda$ and an element $w \in W$. Suppose that the associated intertwining integral 
$
     j_{w, *}^{w \phi, \phi}
$
is clean. Then the corresponding convolution functor 
\[
  j_{w, *}^{w \phi, \phi}:  \gk\mod^{I, \phi} \rightarrow \gk\mod^{I, w \phi} 
\]
is $t$-exact, and for any $\mu \in \phi + \Lambda$ there are isomorphisms 
\begin{equation} \label{e:cleanconv}
      j_{w, *}^{w\phi, \phi} \sI M_{\mu} \simeq M_{ w \mu} \quad \text{and} \quad j_{w, *}^{w\phi, \phi} \sI L_{\mu} \simeq L_{ w \mu}.
\end{equation}

\end{pro}

\begin{re} It formally follows that the convolution also exchanges the contragredient Verma modules.
\end{re}

 \begin{proof}[Proof of Proposition \ref{p:cleanconv}] We first reduce the assertion to the case where $w$ is a simple reflection. If $w$ has length zero, the assertion is clear. If $w$ has length greater than one, then $w > ws$ for some simple reflection $s$, and we have a corresponding factorization 
\begin{equation}  \label{e:conv2step}
    j_{w, *}^{w\phi, \phi} \hspace{1mm} \simeq \hspace{1mm}j_{ws, *}^{w \phi, s\phi} \hspace{1mm}\overset I \star \hspace{1mm} j_{s, *}^{s\phi, \phi}.
\end{equation}

We claim that  both factors of the right-hand side are again clean. We first show this for $j_{s, *}^{s\phi, \phi}$. Indeed, if not, we would have $s\phi = \phi$ and a surjection $j_{s, *}^{\phi, \phi} \twoheadrightarrow j_{e, *}^{\phi, \phi}$. By convolving this with $j_{w, *}^{w \phi, \phi}$, by the right exactness of convolution with costandard objects we would obtain a surjection 
\[
    j_{w, *}^{w \phi, \phi}  \twoheadrightarrow j_{ws, *}^{w \phi, \phi},
\]
which would contradict the cleanness of $j_{w, *}^{w \phi, \phi}$.

To proceed, we will use the following basic lemma.

\begin{lemma} \label{l:convs}Suppose $s$ is a simple reflection of $W,$ and $\phi$ is a character sheaf such that the object $j_{s, *}^{s\phi, \phi}$ is clean. Then for any $w \in W$ there are isomorphisms 
\begin{equation} \label{e:convenienttruth}
     j_{w, !}^{w \phi, s\phi} \sI j_{s,!}^{s\phi, \phi} \hs \simeq \hs j_{ws, !}^{w \phi, \phi} \quad \text{and} \quad  j_{w, *}^{w \phi, s\phi} \sI j_{s,*}^{s\phi, \phi} \hs \simeq \hs j_{ws, *}^{w \phi,  \phi}
\end{equation}
That is, convolution with $j_{s, *}^{s\phi, \phi}$ sends (co)standard objects to (co)standard objects. 
\end{lemma}

\begin{re} This lemma appears, {\em mutatis mutandis}, in the study of finite twisted Hecke categories performed in \cite{ly}, and likely earlier as well. A similar analysis of the  affine twisted Hecke categories, which refines certain results proved in the remainder of the present paper, as well as its applications, e.g. to the tamely ramified local quantum Langlands correspondence, will appear in forthcoming work of the author and Y. W. Li, Z. Yun, and X. Zhu. \end{re}

\begin{proof}[Proof of Lemma \ref{l:convs}] If $w < ws$, then the both assertions of Equation \eqref{e:convenienttruth} are well known to hold without assumptions of cleaness, and follow  from the fact that 
\[
      \mathring{I}wI \overset I \times IsI \rightarrow \mathring{I}wsI
\]
is an isomorphism. The claims for $w > ws$ then follow from the fact that  $ - \sI j_{s, *}^{s\phi,  \phi}$ is inverse to $ - \sI j_{s, !}^{\phi,  s\phi}$, and the cleanness of the latter.  
\end{proof}

\iffalse 
 It is straightforward to see that convolution with a clean object associated to a simple reflection $$ - \hs \overset I \star \hs j_{s, *}^{s\chi, \chi}  : D_\kappa(I, \eta \backslash LG / s \chi) \rightarrow D_\kappa(I, \eta \backslash LG / \chi)$$ sends standard objects to standard objects, and costandard objects to costandard objects.
 \fi 
 
 With the lemma in hand, the asserted cleanness of $j_{ws, *}^{w \phi, s \phi}$ follows from   \eqref{e:conv2step}.  
 
 Iterating the factorization \eqref{e:conv2step}, it follows that $j_{w, *}^{w \phi, \phi}$ may be written as a convolution of clean costandard objects associated to simple reflections and a length zero costandard object. As the claims of the proposition are compatible with composition, this completes the reduction to the case of $w$ being a simple reflection $s$.

  In this case, write $P$ for the corresponding (minimal) standard parahoric subgroup of $LG$. We first argue that for any $\mu \in \phi + \Lambda$ one has 
 \begin{equation} \label{e:convverma}
     j_{s, *}^{s\phi, \phi} \sI M_{\mu} \simeq M_{s \mu}.
 \end{equation}
 To see this, write $M$ and $\mathfrak{m}$ for the corresponding Levi factors of $LG$ and $\gk$, respectively, and note these are of semisimple rank one. By the $D(P)$-equivariance of parabolic induction
 \[
   \on{pind}: \mathfrak{m}\mod \rightarrow \gk\mod,
 \]
see for example Section 2.5.4 of \cite{ahc}, the assertion \eqref{e:convverma} reduces to the analogous assertion for $\mathfrak{m}$. To see the latter, write $B^s$  for the opposite Borel of $P$ containing $T$, and $N^s$ for its prounipotent radical. If we pick any lift $\dot{s}$ of $s$ to $N_{M}(T)$, we may identify $j_{s, *}^{s\phi, \phi}$ with the composition
\begin{equation} \label{e:intop}
    \mathfrak{m}\mod^{B, \phi} \xrightarrow{\on{Av}_{N^s, *}[1]} \mathfrak{m}\mod^{B^s, s \phi} \xrightarrow{ \delta_{\dot{s}} } \mathfrak{m}\mod^{B, s\phi}.   
\end{equation}
The remainder follows from the highest weight representation theory of $\fsl_2$. Briefly, by our assumptions on $\phi$ and $s\phi$, the appearing categories are semisimple, so one is reduced to a straightforward verification at the level of characters.

   It remains to prove the $t$-exactness of $j_{s, *}^{s\phi, \phi}$ and the interchanging of simple modules. But recall that an object $N$ of $\gk\mod^{I, s\phi}$ lies in cohomological degrees $\geqslant 0$ if and only if 
   $$
      \on{Hom}(M_\nu, N) \xrightarrow{\sim} \tau^{\geqslant 0}\on{Hom}(M_\nu, N)$$$ \text{for all $\nu$ in $s\phi + \Lambda$}.
  $
  Therefore, the claimed $t$-exactness of $j_{s, *}^{s \phi, \phi}$ follows from \eqref{e:convverma}. Moreover, as $j_{s, *}^{s \phi, \phi}$ is further an equivalence, it also exchanges the simple quotients of the Verma modules, as desired.  \end{proof}

\subsection{} We next gather some facts which will be of use in controlling the final intertwining integral on the right-hand side of \eqref{e:e3steps}, namely $j_{w_\chi, *}^{\chi, \chi}$. 

 Recall that $W_\chi$ is a Coxeter group, and in particular carries a distinguished set of simple reflections and a Bruhat order. We will need the following assertion. 
\begin{pro} Fix elements $y, s \in W_\chi$ such that $s$ is a simple reflection of $W_\chi$ and $y < ys$. There exists an isomorphism of costandard objects 
\[
       j_{ys, *}^{\chi, \chi} \hs \simeq \hs j_{y, *}^{\chi, \chi} \sI j_{s, *}^{\chi, \chi}.   
\]
\label{p:twistup}
\end{pro}

We emphasize that the appearing simple reflection $s$ of $W_\chi$ typically is not a simple reflection of the ambient group $W$. This renders the proposition not entirely geometrically obvious, as the underlying convolution of Schubert varieties will typically not be an isomorphism. Before giving the proof, we note the following useful consequence. 

\begin{cor}\label{c:littlesteps} For an element $y$ of $W_\chi$,  fix a reduced expression 
\[
      y = s_{i_1} s_{i_2} \cdots s_{i_\ell}. 
\]
There exists a corresponding isomorphism of costandard objects 
\[
     j_{\dot{y}, *}^{\chi, \chi} \hs \simeq \hs j_{s_{i_1}, *}^{\chi, \chi} \hs \overset I \star \hs j_{s_{i_2}, *}^{\chi, \chi} \sI \cdots \sI  j_{s_{i_\ell}, *}^{\chi, \chi}.  
\]
\end{cor}

\begin{proof}[Proof of Proposition \ref{p:twistup}] We will reduce to the case of $s$ being also a simple reflection of $W$ using the following lemma.

\begin{lemma} \label{l:gotosimp}If $s$ is a simple reflection of $W_\chi$, then there exists elements $z$ and $t$ in $W$ such that the following hold. 
\begin{enumerate}
    \item The element $t$ is a simple reflection of $W$.
    \item One has the equality $s = ztz^{-1}$. 
    \item The action of $z$ on the real affine coroots restricts to  an isomorphism  
    \begin{equation}
          z: \check{\Phi}_\chi^+ \xrightarrow{\sim} \check{\Phi}_{z \chi}^+.  \label{e:ziscool}
    \end{equation}
    In particular, conjugation by $z$ yields an isomorphism of Coxeter systems 
    \[
               z: W_\chi \xrightarrow{\sim} W_{z \chi}. 
    \]
\end{enumerate}

\end{lemma}
\begin{proof} If the coroot $\halpha$ corresponding to $s$ is simple in $\check{\Phi}^+$ we are done. If not, there exists a simple reflection $u$ of $W$ such that 
\begin{equation} \label{e:down}
     \halpha > u(\halpha) > 0. 
\end{equation}
Since by assumption $\halpha$ is a simple coroot in $\check{\Phi}^+_\chi$, it follows that $u$ does not lie in $W_\chi$, and in particular we have that acting by $u$ restricts to an isomorphism
\[
    u: \check{\Phi}^+_\chi \xrightarrow{\sim} \check{\Phi}^+_{u\chi}. 
\]
This reduces the assertions of the lemma for the pair $(\chi, s)$ to those for $( u \chi, usu^{-1})$. By Equation \eqref{e:down} and the fact that there are only finitely many positive coroots less than $\halpha$, iterating this argument must bring us to a simple coroot of $\check{\Phi}^+$, as desired. \end{proof}

 Let $s$ be as in the statement of the proposition, and let $t$ and $z$ be as in the statement of Lemma \ref{l:gotosimp}. By the proof of the lemma, in combination with Lemma \ref{l:convs}, it follows that the objects
\[
     j_{z, *}^{\chi, z \chi} \quad \text{and} \quad j_{z^{-1}, *}^{z \chi, z}
\]
are clean. In particular, the conjugation monoidal equivalence 
\[
  j_{z, *}^{\chi, z\chi} \sI - \sI j_{z^{-1}, *}^{z \chi, \chi}: D( I, z \chi \backslash LG / I, z \chi) \simeq D(I, \chi \backslash LG / I, \chi)
\]
reduces the assertion of the proposition to the case where the simple reflection $s$ of $W_\chi$ is a simple reflection of $W$, which is standard, cf. Lemma \ref{l:convs}. 
\end{proof}

Since we will use it later, we explicitly record one fact developed in the proof of Proposition \ref{p:twistup}. Namely, the following lifts Lemma \ref{l:gotosimp}, particularly statement (2) therein,  to the intertwining functors.

\begin{lemma} \label{l:conjsimp}Let $s, t, z$ be as in Lemma \ref{l:gotosimp}. Then there is an isomorphism of costandard objects 
\[
     j_{s, *}^{\chi, \chi} \hs \simeq \hs j_{z, *}^{\chi, z\chi} \sI j_{t, *}^{z\chi, z\chi} \sI j_{z^{-1}, *}^{z \chi, \chi}. 
\]
That is, up to conjugation by a clean object, any costandard object associated to a simple reflection of $W_\chi$ is a costandard object associated to a simple reflection of $W$. 
\end{lemma}

\iffalse

\begin{proof} It follows from Equation \ref{e:ziscool} that the objects
%
\[
     j_{z, *}^{\chi, z \chi} \quad \text{and} \quad j_{z^{-1}, *}^{z \chi, z}
\]
%
are clean, which readily implies the lemma. \end{proof}

We now deduce the statement of the proposition. Indeed, the conjugation 
%
\[
    j_{z, *}^{\chi, z \chi}
\]
%

\fi

\subsection{} In this subsection, which may be skipped by the reader, we use what we have developed so far to obtain bounds on the cohomological amplitude of the spectrally flowed reduction functors. To our knowledge, these are new already for the plus reduction. 

 Let us collect two pieces of notation. First, we denote the standard length function on $W_\chi$ with respect to its Coxeter generators by  $$\ell_\chi: W_\chi \rightarrow \mathbb{Z}^{\geqslant 0}.$$
Second, let $\sC$ and $\sD$ be dg-categories equipped with $t$-structures and a  functor 
 $$F: \sC \rightarrow \sD.$$Fix integers $a$ and $b$ with $a \leqslant b$. Recall that
$F$ is said to have cohomological amplitude at most $[a,b]$ if $F[b]$ is right exact and $F[a]$ is left exact, i.e. 
\[
    F(\sC^{\leqslant 0}) \subset \sD^{\leqslant b} \quad \text{and} \quad F(\sC^{\geqslant 0}) \subset \sD^{\geqslant a}. 
\]
In particular, for any object $\xi$ in the heart of the $t$-structure on $\sC$, $F(\xi)$ has nonzero cohomology, with respect to the truncation functors on $\sD$, in at most cohomological degrees $a \leqslant i \leqslant b$.

\begin{theo} \label{t:cohamp}Let $w, \chi, w_\chi$ be as in Proposition \ref{p:3steps}. Then the spectrally flowed reduction functor 
\[
\Psi^{w}: \gk\mod^{I, \chi} \rightarrow \sW\mod
\]
has cohomological amplitude at most 
\begin{equation} \label{e:cohamp}
[- \ell_\chi(w_\chi), 0].
\end{equation}
\end{theo}
\iffalse 
Before giving the proof, we would like to  

\begin{ex} Taking $w = w_\circ t^{\crho}$, the theorem is equivalent to the assertion that the plus reduction $\Psi^+: \gk\mod^{I, \chi} \rightarrow \sW\mod$ has cohomological amplitude at most
%
\[
   [ -\ell_\chi(w_\chi), 0];
\]
%
for the relation between $\Psi^+$ and $\Psi^{\dot{w}_\circ t^{\crho}}$ cf. Example \ref{e:pred}. \end{ex}
\fi 

\begin{re} The bounds \eqref{e:cohamp} are essentially optimal. On the one hand, the upper bound is always obtained; indeed it follows from Corollary \ref{c:3steps} that $\Psi^{w}$ sends Verma modules to twisted Verma modules concentrated in degree zero. On the other hand, we will see momentarily that reductions of dominant irreducibles, when nonzero, are concentrated in degree $ - \ell_\chi(w_\chi)$. 
\end{re}

\begin{proof}[Proof of Theorem \ref{t:cohamp}] In view of Corollary \ref{c:3steps}, we equivalently must show that the composition 
\[
      \gk\mod^{I, \chi} \xrightarrow{ j_{{w}_\chi, *}^{\chi, \chi} } \gk\mod^{I, \chi} \xrightarrow{ j_{{w}_-, *}^{w_-\chi, \chi}} \gk\mod^{I, w_-\chi} \xrightarrow{\Psi^-} \sW\mod,
\]
has cohomological amplitude at most $[-\ell_\chi(w_\chi), 0]$. Recall that the middle and rightmost functors are $t$-exact, see Proposition \ref{p:cleanconv} and Theorem \ref{t:araminus}, respectively. Therefore, it is enough to show that $j_{{w}_\chi, *}^{\chi, \chi}$ has cohomological amplitude at most $$[-\ell_\chi(w_\chi), 0].$$

To see this, by Corollary \ref{c:littlesteps} it suffices to show that for a simple reflection $s$ of $W_\chi$ the intertwining integral $j_{s, *}^{\chi, \chi}$ has cohomological amplitude at most$$[-1, 0].$$By Lemma \ref{l:conjsimp} and Proposition \ref{p:cleanconv}, we may assume that $s$ is also a simple reflection of $W$. However, this final case is clear, for example by the analog of Equation \eqref{e:intop}.  \end{proof}

\subsection{} To state our main theorem,  we now introduce one final piece of notation, and then review several others. 

Let us write  $\ftp$ for the set of dominant weights at level $\kappa$. This carries an action of $W$, which we denote by
\[
   W \times \ftp \rightarrow \ftp, \quad \quad (w, \lambda) \mapsto w * \lambda, 
\]
defined as follows. Given a pair $(w, \lambda)$, the orbit of $w\lambda$ under its integral Weyl group contains a unique dominant weight $w\lambda^+$, and we set 
\[
     w * \lambda := w \lambda^+. 
\]
It is straightforward that this defines an action. 
%Slightly alternatively, one may use the canonical identification between $D_\kappa$ and the set of integral Weyl group orbits containing dominant weights, and the action of $W$ on the latter. 
%

For the statement of the following theorem we recall that $\pi$ denotes the Harish-Chandra projection
\[
   \pi: \ft^* \rightarrow \ft^* /\hspace{-1.1mm}/ W_f,
\]
and that a weight $\mu \in \ft^*$ is said to be {finite antidominant} if for every finite positive coroot $\halpha$ one has
\[
    \langle \halpha, \lambda \rangle \notin \mathbb{Z}^{> 0},
\]
where the appearing pairing incorporates a $\rho$-shift as in Section \ref{sss:intwylgroup}. Finally, we refer to Section \ref{s:minusred} for our normalization of highest weights for $\sW$-modules. We are now ready to prove our main result.

\begin{theo} \label{t:maintheorem5}Fix an element $w \in \Wf$ and a dominant irreducible module
\[
     L_\lambda \in \gk\mod^{I, \chi}.
\]
Let $w_\chi$ be as in Proposition \ref{p:3steps}. Then the corresponding spectrally flowed Drinfeld--Sokolov reduction is given by  
\[
\Psi^{\dot{w}}(L_\lambda) \simeq \begin{cases} L_{\pi(w * \lambda)}[\ell_\chi(w_\chi)] & \text{if $w * \lambda$ is finite antidominant,}  \\ 0 & \text{otherwise}. \end{cases}
\]

\end{theo}

\begin{proof} We will proceed step by step along the factorization of Corollary \ref{c:3steps}, namely
\[
     \Psi^{w} \simeq \Psi^- \circ j_{w_-, *}^{w_- \chi, \chi} \circ j_{w_\chi, *}^{\chi, \chi}.  
\]

To begin, we claim that one has an isomorphism
\begin{equation} \label{e:1ststep}  j_{w_\chi, *}^{\chi, \chi} \sI L_\lambda \simeq L_{\lambda} [\ell_\chi(w_\chi)]. 
\end{equation}
To see this, by Corollary \ref{c:littlesteps} it is enough to show that for a simple reflection $s$ of $W_\chi$ one has
\[
   j_{s, *}^{\chi, \chi} \sI L_\lambda \simeq L_{\lambda}[1]. 
\]
By Lemma \ref{l:conjsimp} and Proposition \ref{p:cleanconv}, we may further assume that $s$ is a simple reflection of $W$. In this case, let $P$ be the corresponding parahoric subgroup of $LG$, and note that $\chi$ extends uniquely to a character sheaf on $P$, which we denote by the same letter. 

 As $L_\lambda$ is $(P, \chi)$-equivariant, it suffices to see that for any $D_\kappa(LG)$-module $\sC$ the composition 
\begin{equation}  \label{e:1comp}
    \sC^{P, \chi} \xrightarrow{\on{Oblv}} \sC^{I, \chi} \xrightarrow{j_{s, *}^{\chi, \chi}} \sC^{I, \chi}
\end{equation}
is equivalent to 
\begin{equation} \label{e:2comp}
  \sC^{P, \chi} \xrightarrow{\on{Oblv}[1]} \sC^{I, \chi}. 
\end{equation}
However, it is straightforward to see by a calculation on the Levi factor of $P$, which is of semisimple rank one, that both \eqref{e:1comp} and \eqref{e:2comp} are given by convolution with the same object of 
\[
       D_\kappa(I, \chi \backslash LG / P, \chi),
\]
as desired. This completes the proof of \eqref{e:1ststep}. 

To proceed, recall that $j_{w_-, *}^{w_- \chi, \chi}$ is clean. By Proposition \ref{p:cleanconv}, it therefore follows that 
\[
    j_{w_-, *}^{w_-, \chi} \sI L_{ \lambda}[\ell_\chi(w_\chi)] \simeq L_{w_- \lambda }[\ell_\chi(w_\chi)].
\]
Finally, applying the minus reduction $\Psi^-$ to $L_{w_-  \lambda}[\ell_\chi(w_\chi)]$ yields the statement of the theorem. 
\end{proof}

\bibliographystyle{amsalpha}
\bibliography{samplez}

\end{document}